\documentclass[a4paper,twoside,10pt]{article}

\usepackage[english]{babel}
\usepackage[latin1]{inputenc}
\usepackage{lmodern}
\usepackage[T1]{fontenc}
\usepackage{indentfirst}
\usepackage{amsmath,amssymb}

\usepackage[a4paper,top=3cm,bottom=3cm,left=3cm,right=3cm,%
bindingoffset=5mm]{geometry}
\usepackage{lmodern}
\usepackage[T1]{fontenc}
\usepackage{lettrine}
\usepackage{booktabs}

\usepackage{authblk}
\usepackage{emp}
\usepackage{amsthm}
\usepackage{amsmath,amssymb,amsthm}
\usepackage{braket}
\usepackage{chngpage}
\usepackage{cases}
\usepackage{graphicx}
\usepackage{tabularx}
\usepackage{multirow}

\newcommand\RE{\mathbb{R}}

\newcommand\K{\mathbb{K}}
\newcommand\req[1]{~(\ref{eq:#1})}
\newcommand\nulla[1]{}
\newcommand\T{\mathcal{T}}
\newcommand\Vh{V_h^k}

\theoremstyle{remark}
\newtheorem{remark}{Remark}[section]

\theoremstyle{remark}
\newtheorem{test}{Test}[section]

\theoremstyle{plain}
\newtheorem{theorem}{Theorem}[section]

\usepackage{lipsum}

\usepackage{setspace}

\usepackage{setspace}

\author[1]{Francesca Gardini \thanks{francesca.gardini@unipv.it}}
\author[2]{Giuseppe Vacca  \thanks{giuseppe.vacca@unimib.it}}

\affil[1]{Dipartimento di Matematica ``F.~Casorati'',
Universit\`a degli Studi di Pavia, via Ferrata 1, I-27100 Pavia, Italy}
\affil[2]{Dipartimento di Matematica e Applicazioni,  Universit\`a degli Studi di Milan-Bicocca, Via Roberto Cozzi, 55 - 20125 Milano, Italy}

\title{Virtual Element Method for Second Order Elliptic Eigenvalue Problems}

\date{\today}

\begin{document}

\maketitle
\begin{abstract}
We study the virtual element (VEM) approximation of elliptic eigenvalue problems.
The main result of the paper states that VEM provides an optimal order approximation 
of the eigenmodes. A wide set of numerical tests confirm the theoretical analysis.
\end{abstract}


\maketitle

\section{Introduction}
The Virtual Element Method (VEM) is a brand new approximation technique 
introduced in~\cite{vem} which has been applied to several problems.  
In its abstract formulation the method is a generalization of the conforming finite element method 
which allows, nevertheless, the use of general polygonal and polyhedral meshes 
without having to integrate complex non-polynomial functions on the elements. 

The Virtual Element Method has been developed successfully for a large range of problems:  the linear elasticity problems, both for the compressible and the nearly incompressible case \cite{VEM-elasticity, paulinopost}, a stream and a non-conforming formulation
 of VEMs for the Stokes problem \cite{VEM-stream,cangiani2016non}, the non-linear elastic and inelastic deformation problems, mainly focusing on a small deformation regime \cite{BLM15}, the Darcy problem in mixed form \cite{VemMixed}, the plate bending problem \cite{Brezzi:Marini:plates}, the Steklov eigenvalue problem \cite{mora2015virtual}, the general second order elliptic problems in primal \cite{vemgeneral} and mixed form \cite{VEM-mixedgeneral}, the Cahn-Hilliard equation \cite{VEM-cahn}, the Helmholtz problem \cite{VEM-helmholtz},  the discrete fracture network simulations \cite{berrone, berrone2}, the time-dependent diffusion problems \cite{vaccabeirao, vacca2016virtual} and the Stokes problem \cite{VEM-divfree}.  In \cite{VEM-nonconforming,VEM-nonconforming1} the authors present a non-conforming Virtual Element Space. {\itshape{A posteriori}} error estimates are studied in~\cite{VEMaposteriori,cangiani2016posteriori,aposterioristeklov}. 
 $H(div)$ and $H(curl)$ VEM and VEM with arbitrary regularity are presented in~\cite{conforming} and~\cite{BM13}, respectively. 
 In~\cite{hpvem} the VEM $hp$ version is analyzed.
Finally, in \cite{VEM-serendipity, da2016serendipity} the authors introduce the last version of Virtual Element spaces, the Serendipity VEM spaces that, in analogy with the Serendipity FEMs, allows to reduce the number of degrees of freedom.

In this paper we study the Virtual Element Method applied to elliptic eigenvalue problems. 
As a model problem we consider the Laplace eigenvalue problem. Nevertheless the 
analysis generalizes straightforward to the case of more general second order 
elliptic eigenvalue problems. 
The discretization of the 
problem requires the introduction of two discrete bilinear forms, one being the 
approximated grad-grad form and the other being a discrete version of the $L^2$--inner product. 
The latter one is built using the techniques of~\cite{projectors}. In particular, we consider both a 
non-stabilized form and a stabilized one, and we study the convergence properties of 
the corresponding discrete formulations.
It is shown that the Virtual Element Method provides
optimal convergence rates both for the eigenfunctions and the eigenvalues. \\
The paper is organized as follows. In Section~\ref{sc:problem}, we set up the model eigenproblem, 
while in Section~\ref{sc:spaces}  
we recall the definition of the bi-dimensional and three-dimensional Virtual Element Spaces. 
In Section~\ref{sc:approx} we introduce the virtual element formulation of the problem, 
and in Section~\ref{sc:compact} we recall some 
fundamental results for the spectral approximation of compact operators. 
In Section~\ref{sc:analysis} we prove the main results of the paper, which consist in the a optimal {\itshape priori} error estimates for the VEM approximation of general elliptic eigenvalue problems. 
We discuss the implementation details in Section~\ref{se:experiments} and show the behaviour of the method for a set of numerical examples. We finally draw the conclusions in Section~\ref{sc:end}.

\section{Setting of the problem}
\label{sc:problem}

We are interested in the problem of computing the eigenvalues of the Laplace operator, namely \\ 
find $\lambda\in\RE$ such that there exists $u$, with $\|u\|_0=1$ satisfying
\begin{equation}\label{eq:eigPbmstrong}
\left\{
\begin{array}{ll}
-\Delta u=\lambda u\quad & \text{in }\Omega\\
u=0 & \text{on }\Gamma,
\end{array}
\right.
\end{equation}
where $\Omega\subset\RE^n\ (n=2,3)$ 
is a bounded polygonal/polyhedral domain with Lipschitz boundary $\Gamma$. 
 
For ease of exposition, we focus on the case of Dirichlet boundary conditions. The extension to other boundary conditions are analogous.

The variational formulation of problem\req{eigPbmstrong} reads:\\
\text{find  $\lambda\in\RE$ such that there exists $u\in V$, with $\|u\|_0=1$ satisfying}
\begin{equation}\label{eq:eigPbm}
a(u,v)=\lambda b(u,v)\quad\forall v\in V,
\end{equation}
where $V=H^1_0(\Omega)$, 
$a(u,v)=\int_{\Omega}\nabla u\cdot\nabla v$, and $b(\cdot,\cdot)$ denotes the 
$L^2$-inner product.\\
It is well-know that the eigenvalues of problem\req{eigPbm} form a positive increasing divergent sequence, 
and that the corresponding eigenfunctions are an orthonormal basis of $V$ with respect both to the $L^2$-inner product  
and to the scalar product associated with the bilinear form $a(\cdot,\cdot)$.\\
Due to regularity results~\cite{agmon}, 
there exists a constant $r > 1/2$ depending on $\Omega$, such that the solution $u$ 
belongs to the space $H^{1+r}(\Omega)$. It can be proved that $r$ is at least one if $\Omega$ is a convex domain, while $r$ 
is at least $\pi/\omega - \varepsilon$ for any $\varepsilon >0$ for a non--convex domain, being $\omega<2\pi$ the maximum interior angle of $\Omega$.

We will also need the source problem associated with the eigenvalue problem\req{eigPbm}: 
given $f\in L^2(\Omega)$, find $u^s\in V$ such that
\begin{equation}\label{eq:sourcePbm}
a(u^s,v)=b(f,v)\quad\forall v\in V.
\end{equation}

Throughout the paper, we will make use of the following notation. 
We will denote by $|\cdot|_{s,\omega}$ and $\|\cdot\|_{s,\omega}$ the seminorm and the norm in the Sobolev 
space $H^s(\omega)$, respectively, while $(\cdot,\cdot)_{\omega}$ will denote le $L^2$-inner product over 
the domain $\omega$. 
Moreover, if $\omega = \Omega$, the subscript $\omega$ may be omitted. 
For a positive integer $k$, $\mathbb{P}_k(\omega)$ 
will denote the space of polynomials on $\omega$ of degree at most $k$. Finally, $a^{\omega}(\cdot,\cdot)$ 
and $b^{\omega}(\cdot,\cdot)$
will denote respectively the restriction of the form $a(\cdot,\cdot)$  and $b(\cdot,\cdot)$ on $\omega$.

\section{Virtual Element Spaces} 
\label{sc:spaces}

In this section, we briefly recall the definition of the Virtual Element Spaces.  
 We present separately the the bi-dimensional and the three-dimensional case.

\subsection{Bi-dimensional case}
\label{sub:2d}
Let $\{\mathcal{T}_h\}_h$ be a sequence of decompositions of $\Omega$ into polygons $P$, and let 
$\mathcal{E}_h$ denote the set of edges $e$ of $\mathcal{T}_h$. For every element $P$, we denote by $|P|$ its area and 
by $h_P$ its diameter. Similarly, for each edge $e$, $|e|$ or, equivalentely, $h_e$ will denote its length. Depending on the context, $\partial P$ 
may denote the boundary of element $P$ or the set of the element edges. As usual, the mesh size $h$ is the maximum diameter of the elements $P$ in $\mathcal{T}_h$.\\
In accordance with~\cite{vem}, we assume the following mesh regularity condition: there exist a positive constant 
$\gamma$, independent of $h$, such that 
each element $P\in\T_h$ is star-shaped with respect to a ball of radius greater than $\gamma h_P$, and 
for every element $P$, and for every edge $e\subset\partial P$, $h_e\ge\gamma h_P$.

Following~\cite{vem,projectors}, for every integer $k\ge 1$ and for every element $P\in\mathcal{T}_h$ we define
$$
\tilde{V}_h^k(P):=\{v\in C^0(\partial P):  v|_e\in\mathbb{P}_k(e)\ \forall e\subset\partial P, \Delta v\in\mathbb{P}_k(P)\}.
$$ 
 For each $v$ in $\tilde{V}_h^k(P)$, we consider the  
following linear operators $\mathbf{D^{2D}}$ split into three sets:
\begin{itemize}
\item $\mathbf{D_1^{2D}}$: the values $v(V_i)$ at the vertices $V_i$ of $P$,

\item $\mathbf{D_2^{2D}}$: the scaled edge  moments up to order $k-2$ 
$$
\dfrac{1}{|e|}\displaystyle{\int_e} v\,m\,\text{d}s,\quad \forall m\in\mathcal{M}_{k-2}(e), \quad \text{on each edge $e$ of $P$,} 
$$
\item $\mathbf{D_3^{2D}}$: the scaled element moments up to order $k-2$ 
 $$\dfrac{1}{|P|}\displaystyle{\int_P} v\,m\,\text{d}\mathbf{x}, \quad \forall m\in\mathcal{M}_{k-2}(P),$$\\
\end{itemize}
where $\mathcal{M}_{k-2}(\omega)$ denotes the set of scaled monomials on $\omega$
$$
\mathcal{M}_{k-2}(\omega)=\Big\{ \Big( \dfrac{\mathbf{x}-\mathbf{x}_{\omega}}{h_{\omega}} \Big)^{\mathbf{s}}, |\mathbf{s}|\le k-2 \Big\}
$$
being $\mathbf{x}_{\omega}$ the barycenter of $\omega$, and where we set 
$\mathcal{M}_{-1}(\omega)= \emptyset$.

\begin{remark}
It can be proved that each scaled monomial in $\mathcal{M}_k(\omega)$, and also the 
linear operators  $\mathbf{D^{2D}}$, scales like $1$ 
(cf.~\cite{VemGuide} Remark 1.1., and Remark 2.5.)
\end{remark}

\begin{remark}
The  linear operators $\mathbf{D_3^{2D}}$ allows to exactly compute 
the $L^2(P)$-projection of any function in $\tilde{V}_h^k(P)$ onto the local space of piecewise polynomial of 
degrees at most $k-2$.  Indeed, given a function $v\in \tilde{V}_h^k(P)$, 
the $L^2(P)$-projection of $v$ is defined as the unique element $\Pi^P_{k-2}v\in\mathbb{P}_{k-2}(P)$ such that
\begin{equation}
\label{eq:proj}
\int_P( \Pi_{k-2}^Pv)\,p\,\text{d}\mathbf{x}=\int_Pv\,p\,\text{d}\mathbf{x}\quad\text{forall }p\in\mathcal{M}_{k-2}(P).
\end{equation}
Notice that~\req{proj} is a linear system with right hand side given by the  $\mathbf{D_3^{2D}}(v)$. 

We observe that, instead, the linear operators $\mathbf{D_3^{2D}}$ are not enough to compute the projection onto the space 
of piecewise polynomial of degree $k$ .
\end{remark}

From the linear operators $\mathbf{D^{2D}}$, on each element $P$ we can construct and exactly compute a projection operator
$\Pi_k^{\nabla}: \tilde{V}_h^k(P)\rightarrow\mathbb{P}_k(P)$ defined as follows:
\begin{equation}
\label{eq:pinabla}
a^P(\Pi_k^{\nabla}v-v, p)=0\quad\forall p\in\mathbb{P}_k(P)
\end{equation}
and 
\begin{equation}
\int_{\partial P}(\Pi_k^{\nabla}v-v)\text{d}s=0\quad\text{for } k=1
\end{equation}
or
\begin{equation}
\int_P (\Pi^{\nabla}_k v-v)\text{d}\mathbf{x}=0\quad\text{for } k\ge2.
\end{equation}
We observe that this operator is well-defined also for functions in $H^1(P)$, but in this case it is not exactly computable. 
On the other hand, for all $v\in\tilde{V}_h^k$, $\Pi_k^{\nabla}v$ can be computed 
only in terms of $\mathbf{D^{2D}}(v)$.

The local virtual space is then defined as
\begin{equation}
\label{eq:virtual2d}
V_h^k(P):=\Big\{v\in\tilde{V}_h^k(P): \int_Pv\,p\,\text{d}{\bf{x}}=\int_P (\Pi_k^{\nabla}v)p\,\text{d}{\bf{x}}\ 
\forall p\in(\mathbb{P}_k / \mathbb{P}_{k-2}(P))\Big\},
\end{equation}
where $(\mathbb{P}_k / \mathbb{P}_{k-2}(P))$ denotes the space of polynomials in $\mathbb{P}_k(P)$
$L^2$--orthogonal to all polynomials in $\mathbb{P}_{k-2}(P)$.\\
We recall that, by construction, the local space $V_h^k(P)$ enjoys the following fundamental properties (see~\cite{projectors}):
\begin{itemize}
\item $\mathbf{(P1)}$ the space $\mathbb{P}_k(P)\subset V_h^k(P)$. This property will guarantee the optimal order of approximation
\item  $\mathbf{(P2)}$  the set of linear operators $\mathbf{D^{2D}}$ constitute a set of degrees of freedom (DoFs) for the space $V_h^k(P)$
\item $\mathbf{(P3)}$ since $V_h^k(P)\subset\tilde{V}_h^k(P)$, the operator $\Pi_k^{\nabla}$ is well-defined on $V_h^k(P)$ and it is still computable in terms of the degrees of freedom
\item $\mathbf{(P4)}$ the standard $L^2$-projection operator $\Pi^0_{k}:V_h^k(P)\rightarrow\mathbb{P}_k(P)$ is computable only in terms of the degrees of freedom
\item $\mathbf{(P5)}$ for all $v\in V_h^k(P)$ the vectorvalued function $\Pi^0_{k-1}\nabla v_h$ can be explicitly computed from the degrees of freedom, see~\cite{vemgeneral}. 
\end{itemize}

The global bi-dimensional discrete space is hence defined in the finite element way as 
$$
V_h^{k,2\rm{D}}=\{v\in V: v|{_P}\in V_h^k(P)\quad\forall P\in\mathcal{T}_h\}.
$$

\subsection{Three-dimensional case}
The aim of this section is to briefly present  the extension of the Virtual Element spaces to the three-dimensional case, recalling  from \cite{projectors} the core idea of the three-dimensional VEM.

Let $\{\mathcal{T}_h\}_h$ be a sequence of decompositions of $\Omega$ into general polyhedral elements $P$. We assume that for all
$h$, each element $P \in \mathcal{T}_h$ fulfils the following assumptions: there exists a uniform positive constant $\gamma$  such that $P$ is star-shaped with respect to a sphere of radius greater than  $\gamma\, h_P$, and every face $f$ of $P$ is star-shaped with respect to a ball of radius greater than $\gamma \, h_f$, and
for every face $f$ of $P$ and for every edge $e$ of $f$, it holds that
$h_e \geq \gamma \, h_f \geq \gamma^2 \, h_P$,
where $h_f$ (resp. $h_e$) denotes the diameter of the face $f$ (resp. the length of the edge $e$).

Let $P$ in $\mathcal{T}_h$. We start by  defining the  virtual local-boundary space, observing that each face $f \in \partial P$ is a polygon. Let us define 
the following space
\begin{equation}
\label{eq:boundary3d}
\mathcal{B}_h^k(\partial P) := \{ v \in C^0(\partial P):  v|_f \in V_h^{k,\rm{2D}}(f)\ \forall f\subset\partial P\}.
\end{equation}

The above space is made of functions that on each face are two-dimensional
virtual functions, that glue continuously across edges. Once the boundary space is defined, the steps to follow in order to define the
local virtual space on $P$ become very similar to the two dimensional case. We
first introduce a preliminary local virtual element space on $P$
$$
\tilde{V}_h^k(P):=\{v\in H^1(P):  v|_{\partial P}\in\mathcal{B}_h^k(\partial P)\ , \Delta v\in\mathbb{P}_k(P)\}.
$$ 
Therefore, extending to the polyhedra the definition~$\eqref{eq:virtual2d}$, we can define the local virtual space
\begin{equation}
\label{eq:virtual3d}
V_h^k(P):=\Big\{v\in\tilde{V}_h^k(P): \int_P v\,p\,\text{d}{\bf{x}}=\int_P (\Pi_k^{\nabla}v)p\,\text{d}{\bf{x}}\ 
\forall p\in(\mathbb{P}_k / \mathbb{P}_{k-2}(P))\Big\},
\end{equation}
Now the degrees of freedom for the space $V_h^k(P)$ are the obvious three-dimensional counterpart of the DoFs of the bi-dimensional case. Let us define the linear operators $\mathbf{D^{3D}}$ split into three sets:
\begin{itemize}
\item $\mathbf{D_1^{3D}}$: the values $v(V_i)$ at the vertices $V_i$ of $P$,
\item $\mathbf{D_2^{3D}}$: the scaled edge  moments up to order $k-2$ 
$$
\dfrac{1}{|e|}\displaystyle{\int_e} v\,m\,\text{d}s,\quad \forall m\in\mathcal{M}_{k-2}(e), \quad \text{on each edge $e$ of $\partial P$,} 
$$
\item $\mathbf{D_3^{3D}}$: the scaled face moments up to order $k-2$ 
 $$\dfrac{1}{|f|}\displaystyle{\int_f} v\,m\,\text{d}\mathbf{x}, \quad \forall m\in\mathcal{M}_{k-2}(f), \quad \text{on each face $f$ of $\partial P$,} $$
\item $\mathbf{D_4^{3D}}$: the scaled element moments up to order $k-2$ 
 $$\dfrac{1}{|P|}\displaystyle{\int_K} v\,m\,\text{d}\mathbf{x}, \quad \forall m\in\mathcal{M}_{k-2}(P).$$  
\end{itemize}
From \cite{projectors} we have that the three-dimensional space $V_h^k(P)$ matches the three-dimensional counterpart of the properties $(\mathbf{P1})$, $(\mathbf{P2})$, $(\mathbf{P3})$, $(\mathbf{P4})$, $(\mathbf{P5})$.

Finally, the three-dimensional global virtual space $V_h^{k,\rm{3D}}$ is defined by using a standard
assembly procedure as in finite elements
$$
V_h^{k,3\rm{D}}=\{v\in V: v|{_P}\in V_h^k(P)\quad\forall P\in\mathcal{T}_h\}.
$$

\section{Virtual Element discretization}
\label{sc:approx}

This section is devoted to the virtual element discretization of the source and the eigenvalue problem. 
We underline that the analysis holds both in the two dimensional and the three dimensional case. 
Therefore, from now on, we do not make any distinction between the spaces in 
two and three dimensions,  
and we simply denote by $V_h^k$ the global VEM space of order $k$.

 The Virtual Element discretization of source problem\req{sourcePbm} reads
\begin{equation}
\label{eq:discreteSource}
\begin{cases}
\text{find } u_h^s\in V_h^k \text{ such that}\\
a_h(u_h^s,v_h)=\langle f_h,v_h \rangle\quad\forall v_h\in V_h^k,\\
\end{cases}
\end{equation}
where $\langle\cdot,\cdot\rangle$ denotes the duality pairing in $V_h^k$, and $f_h\in (V_h^k)'$. 
In particular, $$\langle f_h,v_h\rangle=\sum_{P\in\mathcal{T}_h}(\Pi^0_{k}f,v_h)_P.$$
The discrete bilinear form $a_h(\cdot,\cdot)$ splits as
\begin{equation}
a_h(u_h,v_h)=\displaystyle{\sum_{P\in\mathcal{T}_h}}a_h^P(u_h,v_h).
\end{equation}
with
\begin{equation}
\label{eq:discreteforms}
a_h^P(u_h,v_h)=a^P(\Pi_k^{\nabla} u_h, \Pi^{\nabla}_k v_h) 
+ S^P\Big((I-\Pi_k^{\nabla})u_h,(I-\Pi_k^{\nabla})v_h\Big).
\end{equation}
where $S^P(\cdot,\cdot)$ denotes any symmetric positive definite bilinear form on the element $P$ such that 
there exist two uniform positive constants $c_0$ and $c_1$ such that
$$
c_0a^P(v,v)\le S^P(v,v)\le c_1 a^P(v,v)\quad \forall v\in V_h^k(P) {\text{ with }}\Pi^{\nabla}_kv=0.
$$

\begin{remark}
\label{eq:scale1}
The above requirement means that the form $S^P(\cdot,\cdot)$ scales as $a^P(\cdot,\cdot)$, namely $S^P(\cdot,\cdot)\simeq h_P^{n-2}$, 
with $n=2$ in the bi-dimensional case and $n=3$ in the three-dimensional one.
\end{remark}

The choice of the discrete form $a_h(\cdot,\cdot)$ is driven by the need to satisfy the {\itshape{k-consistency}} and $stability$ properties, $i.e.$
\begin{itemize}
\item {\itshape $k$-consistency}: for all $v\in V_h^k$ and for all $p\in\mathbb{P}_k(P)$ it holds
$$
a_h^P(v,p)=a^P(v,p)
$$
\item {\itshape stability}: there exists two positive constants $\alpha_*,\ \alpha^*$, independent of $h$ and of $P$, such that 
$$
\alpha_*a^P(v,v)\le a_h^P(v,v)\le\alpha^*a^P(v,v)\quad\forall v\in\Vh.
$$
\end{itemize} 
In particular, the first term in\req{discreteforms} ensures $k$--$consistency$, while the second one 
{\itshape stability}. 

The following interpolation and approximation properties hold \cite{cangiani2016posteriori, brsc}. 
\begin{theorem}
\label{thm:interp}
There exists a constant $C$, depending only on the polynomial degree $k$ and the shape regurality $\gamma$, such that 
for every $s$ with $2\le s \le k+1$, for every $h$, for all $P\in\mathcal{T}_h$, 
and for every $w\in H^s(\Omega)$ there exists a 
$w_I\in V_h^k$ such that
\begin{equation}
\|w-w_I\|_{0,P}+h_P |w-w_I|_{1,P}\le Ch_P^s|w|_{s,P}.
\end{equation}
\end{theorem}

\begin{theorem}
\label{thm:approx}
There exists a constant $C$, depending only on the polynomial degree $k$ and the shape regurality $\gamma$, such that 
for every $s$ with $1\le s \le k+1$ and for every $w\in H^s(\Omega)$ there exists a 
$w_{\pi}\in\mathbb{P}_k(P)$ such that
\begin{equation}
\|w-w_{\pi}\|_{0,P}+h_P |w-w_{\pi}|_{1,P}\le Ch_P^s|w|_{s,P}.
\end{equation}
\end{theorem}

In~\cite{projectors} it has been proved that the discrete problem\req{discreteSource} is well-posed and that 
the following optimal {\itshape{a priori}} error estimate holds.
\begin{theorem}
\label{thm:aprioriestimate}
Let $u$ be the solution of problem\req{sourcePbm} and $u_h\in\Vh$ be the solution of the discrete problem\req{discreteSource}, then for every approximation $u_I\in V_h^k$ of u and for every approximation $u_{\pi}$ 
of u that is piecewise in $\mathbb{P}_k$ it holds

\begin{equation}
\label{eq:errorestimates}
|u-u_h|_1\le C (|u-u_I|_1 + |u-u_{\pi}|_{1,h} + \mathcal{F}_h), 
\end{equation}
where $C$ is a positive constant depending only on $\alpha_*$ and $\alpha^*$, 
and for every h, $\mathcal{F}_h(\equiv \|f-f_h\|_{V'_h})$ is the smallest constant such that 
$$
b(f,v_h)-<f_h,v_h>\le\mathcal{F}_h|v_h|_1\quad\forall v_h\in V_h^k.
$$

\end{theorem}

\begin{remark}
We observe that the same result holds also for 
general linear second order elliptic problems, 
provided the form $a(\cdot,\cdot)$ is choosen as in~\req{ahgeneral}. 
 We refer to~\cite{vemgeneral} for an exaustive analysis. 
\end{remark}

\begin{remark}
\label{eq:estimate}
The interpolation and the approximation estimates in Theorems~\ref{thm:interp} and~\ref{thm:approx}, 
the definition of $<f_h,v>$, and the stability of the continuous source problem 
yield that, for $f\in L^2(\Omega)$, the a priori error estimate in Theorem~\ref{thm:aprioriestimate} 
becomes  
\begin{equation}
|u-u_h|_1\le C (h^t|u|_{1+t} + h\|f\|_0)\le Ch^t\|f||_0,
\end{equation}
where $t=\min\{k,r\}$, being $k$ the polynomial degree and $r$ the regularity index of the solution $u$.
\end{remark}

We are now ready to write the VEM approximation of problem\req{eigPbm}:\\
find $\lambda_h\in\RE$ such that there exists $u_h\in V_h^k$, with $\|u_h\|_0=1$ satisfying
\begin{equation}\label{eq:discreteEigPbm}
a_h(u_h,v_h)=\lambda_h b_h(u_h,v_h)\quad\forall v_h\in V_h^k,
\end{equation}
where 
$b_h(\cdot,\cdot)=\sum_{P\in\mathcal{T}_h} b_h^P(\cdot,\cdot)$ 
is a symmetric bilinear forms defined on $V_h^k \times V_h^k$.\\
Two possible choices for the discrete form $b_h(\cdot,\cdot)$ are available. 
The first one is inspired by the virtual approximation 
of the load term in the source problem\req{discreteSource} and reads as follows:
\begin{equation}
\label{eq:discreteb}
b_h^P(u_h,v_h)=\int_P\Pi^0_{k}u_h\Pi^0_{k}v_h \, \text{d}\mathbf{x}.
\end{equation}
The second possible choice consists in considering a discrete bilinear form $\tilde{b}_h(\cdot,\cdot)$ which 
enjoys not only the $k$--$consistency$ property, but also the $stability$ one. 
In this case, as done for the discrete form $a_h(\cdot,\cdot)$, we define 
\begin{equation}
\label{eq:discretebstab}
\tilde{b}_h^P(u_h,v_h)=\int_P\Pi^0_ku_h\Pi^0_kv_h \, \text{d}\mathbf{x} + \tilde{S}^P\Big((I-\Pi_{k}^0)u_h,(I-\Pi_{k}^0)v_h\Big),
\end{equation}
where $\tilde{S}^P$ is any positive definite bilinear form on the element $P$ such that there exist two uniform positive 
constants $\tilde{c}_0$ and $\tilde{c}_{1}$ such that
$$
\tilde{c}_0 \, b^P(v,v)\le\tilde{S}^P(v,v)\le\tilde{c}_{1} \, b^P(v,v)\quad\forall v\in V_h^k(P) \text{ with }\Pi_k^0v=0.
$$

\begin{remark}
\label{eq:scale2}
In analogy with the condition on the form $S^P(\cdot,\cdot)$, 
we require that the form $\tilde{S}^P(\cdot,\cdot)$ scales like $b^P(\cdot,\cdot)$, that is $\tilde{S}^P(\cdot,\cdot)\simeq h^n$, 
with $n=2$, and $n=3$ in the bi-dimensional and in the three-dimensional case, respectively. 
\end{remark}

\begin{remark}
In the definition of the discrete bilinear forms $b_h(\cdot,\cdot)$ and $\tilde{b}_h(\cdot,\cdot)$, 
we project onto the space $\mathbb{P}_k(P)$ since it has been numerically observed that this gives more accurate  results. 
For sure, this choice does not provide a better convergence rate, due to the {\itshape{k-consistency}} property.
\end{remark}

The second VEM approximation of problem~\req{eigPbm} then reads as\\ 
find $\tilde{\lambda}_h\in\RE$ such that there exists $\tilde{u}_h\in V_h^k$, with $\|\tilde{u}_{h}\|_0=1$ satisfying
\begin{equation}\label{eq:discreteEigPbm2}
a_h(\tilde{u}_h,v_h)=\tilde{\lambda}_h \tilde{b}_h(\tilde{u}_h,v_h)\quad\forall v_h\in V_h^k.
\end{equation}

In what follows, we will also need the discrete source problem corresponding 
to the second discrete formulation~\req{discreteEigPbm2}, which reads as:
\begin{equation}
\label{eq:discreteSource2}
\begin{cases}
\text{find } \tilde{u}_h^s\in V_h^k \text{ such that}:\\
a_h(\tilde{u}_h^s,v_h)=\tilde{b}_h(f,v_h)\quad\forall v_h\in V_h^k.\\
\end{cases}
\end{equation}

The well--posedness of the discrete formulation~\req{discreteSource2} stems from that of 
the discrete formulation~\req{discreteSource}, since the bilinear form $a_h(\cdot,\cdot)$ 
is coercive (due to the {\itshape stability} property). 

Summarizing, we consider two different discrete approximation of problems~\req{eigPbm} 
and~\req{sourcePbm}, with at the right hand side a non--stabilized bilinear form  
and a stablized one, respectively.

\section{Spectral approximation for compact operators}
\label{sc:compact}

In this section, we briefly recall some spectral approximation results 
that can be deduced from~\cite{babuskaosborn, boffiActa, kato}. 
For more general results, we refer to the original papers.

Before stating the spectral approximation results, we introduce a natural compact operator associated 
with problem\req{eigPbm} and its discrete counterpart and we recall their connection with the eigenmode convergence.

Let $T\in\mathcal{L}(L^2(\Omega))$ be the solution operator associated with problem\req{eigPbm}, namely 
$T:L^2(\Omega) \rightarrow L^2(\Omega)$ is defined by
 $$
 \left\{
\begin{array}{l}
\label{eq:T}
Tf\in V {\text{ such that}}\\
a(Tf,v)=b(f,v)\quad\forall v\in V.
\end{array}
\right.
$$
Operator $T$ is self-adjoint and positive definite. Moreover, operator $T$ is also compact due to the compact embedding of $H^1(\Omega)$ into $L^2(\Omega)$.

Similarly, let $T_h\in\mathcal{L}(L^2(\Omega))$ be the discrete solution operator 
associated with problem\req{discreteSource} defined as
$$
\left\{
\begin{array}{l}
\label{eq:Th}
T_hf\in V {\text{ such that}}\\
a_h(T_hf,v_h)=b_h(f_h,v_h)\quad\forall v_h\in V_h^k.
\end{array}
\right.
$$
Analougosly,  the discrete solution operator $\tilde{T}_h\in\mathcal{L}(L^2(\Omega))$
associated with problem\req{discreteSource2} is defined as
$$
\left\{
\begin{array}{l}
\label{eq:Thtilde}
\tilde{T}_hf\in V {\text{ such that}}\\
a_h(\tilde{T}_hf,v_h)=\tilde{b}_h(f,v_h)\quad\forall v_h\in V_h^k.
\end{array}
\right.
$$

Operators $T_h$ and $\tilde{T}_h$ are self-adjoint and compact since their ranges are finite dimensional. \\
Finally, the eigensolutions of the continuous and the discrete problems\req{eigPbm} and\req{discreteEigPbm} are respectively related to the eigenmodes of operators $T$ and $T_h$ in the sense that the corresponding eigenvalues are inverse of each other and their eigenspaces coincide. By virtue of this correspondence, the convergence analysis can be derived from the spectral approximation theory for compact operators.\\
Since similar considerations hold for the eigenmodes of operators $T$ and $\tilde T_h$, in the following 
we present only the results relative to operators $T$ and $T_h$.

A sufficient condition for the correct spectral approximation of a compact operator $T$ is the uniform convergence to $T$ of the family of discrete operators $\{T_h\}_h$~\cite{babuskaosborn, boffiActa}:
\begin{equation}
\label{eq:unifconv}
\|T-T_h\|_{\mathcal{L}(L^2(\Omega))}\to 0, \quad \text{as } h\to 0,
\end{equation}
or, equivalentely, 
\begin{equation}
\label{eq:unifconv2}
\|Tf-T_hf\|_0\le C\rho(h)\|f\|_0\quad\forall f\in L^2(\Omega),
\end{equation}
with $\rho(h)$ tending to zero as $h$ goes to zero. 

We remark that~\req{unifconv}, besides the convergence of the eigenmodes, contains also the 
information that no spurious eigenvalues pollute the spectrum.
In fact,
\renewcommand{\theenumi}{\roman{enumi}}
\renewcommand{\labelenumi}{(\theenumi)}
\begin{enumerate}
\item each continuous eigenvalue is approximated by a number of discrete eigenvalues 
  (counted with their multiplicity) that corresponds exactly to its multiplicity;
\item each discrete eigenvalue approximates a continuous eigenvalue.
\end{enumerate}

Since operator $T$ is compact and self-adjoint,
condition\req{unifconv} is also necessary for the correct spectral
approximation; see~\cite{bbg3}.
Regarding the rate of convergence of eigenvalues and eigenvectors, we
refer to~\cite{desclouxnassifrappazI,desclouxnassifrappazII}.

A simple way to estimate the norm of the difference $T-T_h$ is to use
{\itshape{a priori}} error estimates.

\section{Convergence analysis of the method}
\label{sc:analysis}

In this section we study the convergence of the discrete eigenmodes provided by 
the VEM approximation to the continuous ones. We will consider the 
non--stabilized discrete formulation~\req{discreteEigPbm} and the 
stabilized one~\req{discreteEigPbm2} separately. 

\subsection{Convergence analysis for the first formulation}
\label{sub:convnonstab}
In the case of the first VEM approximation of problem~\req{eigPbm}, which corresponds to the choice 
of a non--stabilized $b_h(\cdot,\cdot)$ form, 
the uniform convergence of the sequence of operators ${T_h}$ to $T$ directly stems from 
the \emph{a priori} error estimates in Remark~\req{estimate}. 
The optimal rate of convergence of the eigenfunctions 
and the double rate of convergence of the eigenvalues  
can then be proved following the arguments in~\cite{eig-plates}, Sections 4.

The following theorem ensures the convergence of eigenmodes.
\begin{theorem}
\label{eq:thm1}
The family of operators $T_h$ associated with problem~\req{discreteSource} 
converges uniformly to the operator $T$ associated with problem~\req{sourcePbm}, that is,
\begin{equation}
\label{eq:uniformconv}
\|T-T_h\|_{\mathcal{L}(L^2(\Omega))}\to 0\ for\ h\to 0.
\end{equation}
Let $\lambda$ be an eigenvalue of problem~\req{eigPbm}, with multiplicity $m$, and denote 
the corresponding eigenspace by $\mathcal{E}_{\lambda}$. Then  
exactly $m$ discrete eigenvalues 
$\lambda_{1,h},\cdots,\lambda_{m,h}$, which are repeated according to their respective multiplicities, converge to 
$\lambda$. Moreover, let $\mathcal{E}_{\lambda,h}$ be the direct sum of the eigenspaces 
corresponding to the eigenvalues $\lambda_{1,h},\cdots,\lambda_{m,h}$. Then, there exists a positive number $h_0$ such that for $h\le h_0$ 
the following inequalities are true:
\begin{equation}
\begin{split}
& |\lambda-\lambda_{i,h}|\le Ch^{2t}\quad\forall i=1,\cdots,m,\\
& \hat{\delta}(\mathcal{E}_{\lambda},\mathcal{E}_{\lambda,h})\le Ch^t\|f||_0,
\end{split}
\label{eq:eigmodeconv}
\end{equation}
where the non-negative constant C is independent of $h,\  t = \min\{k, r\},$ 
being $k$ the order of the method and $r$ 
the regurality index of the eigenfunction, and 
$\hat{\delta}(\mathcal{E}_{\lambda},\mathcal{E}_{\lambda,h})$ denotes the 
gap between $\mathcal{E}_{\lambda}$ and $\mathcal{E}_{\lambda,h}$.
\end{theorem}

\begin{proof}
The uniform convergence of $T_h$ to $T$  directly stems from the 
{\itshape a priori} error estimate in Remark~\req{estimate}. Indeed, denoting by $u^s$ and $u_h^s$, respectively, the solutions of the continuous and discrete 
source problems~\req{sourcePbm} 
and~\req{discreteSource} 
corresponding to $f\in L^2(\Omega)$, it holds 
$$
\|T-T_h\|_{\mathcal{L}(L^2(\Omega))}=
\sup_{f\in L^2(\Omega)}\dfrac{\|Tf-T_hf\|_0}{\|f\|_0}
=\sup_{f\in L^2(\Omega)}\dfrac{\|u^s-u_h^s\|_0}{\|f\|_0}\le Ch^l,
$$
with $l=\min\{t,1\}$, being $t=\min\{k,r\}$. 
The eigenmodes convergence~\req{eigmodeconv} can then be proved following step by step 
the lines of the proof of Theorems 4.2. and 4.3. in~\cite{eig-plates}, substituting 
the projector $\Pi^{\Delta_K}$ with $\Pi^0_k$.
\end{proof}

\subsection{Convergence analysis for the second formulation}
\label{sub:convstab}
The convergence analysis of the second discrete formulation of problem~\req{eigPbm}, 
corresponding to the choice of the stabilized form $\tilde{b}_h(\cdot,\cdot)$, is more involved. 
In this case, 
we resort to the abstract theory of the spectral approximation for non-compact operators by Descloux, Nassif, and Rappaz 
(see~\cite{desclouxnassifrappazI, desclouxnassifrappazII}). 

We recall the main convergence theorem stated in~\cite{desclouxnassifrappazII}. 

\begin{theorem}
\label{thm:nassif}
Assume that the following two conditions are satisfied:\\
\begin{equation*}
\textrm{\bf{(1):}}\   \|(T-\tilde{T}_h)_{| V^k_h}\|_{\mathcal{L}(L^2(\Omega))}\to 0\quad {\textrm{\bf (2):}}\ \lim_{h\to 0}\inf_{v_h\in V^k_h}\|v-v_h\|_V=0\ \forall v\in V. 
\end{equation*}
Then the eigenmodes convergence holds.
\end{theorem}

\begin{theorem}
\label{thm:p1p2}
The following two conditions hold true. 
\begin{equation}
\label{eq:uniform}
\textrm{\bf{(1):}}\   \|(T-\tilde{T}_h)_{| V^k_h}\|_{\mathcal{L}(L^2(\Omega))}\to 0\quad {\textrm{\bf (2):}}\ \lim_{h\to 0}\inf_{v_h\in V^k_h}\|v-v_h\|_V=0\ \forall v\in V. 
\end{equation}
\end{theorem}

\begin{proof}
Property {\bf (2)} directly stems from the approximation properties of the virtual element space~(Theorem~\ref{thm:approx}). 
On the other hand, property ${\bf{(1)}}$ can be proved as follows. 
$$
\|(T-\tilde{T}_h)|_{V^k_h}\|_{\mathcal{L}(L^2(\Omega))}\le\|(T-T_h)|_{V^k_h}\|_{\mathcal{L}(L^2(\Omega))}
+\|(T_h-\tilde{T}_h)|_{V^k_h}\|_{\mathcal{L}(L^2(\Omega))}.
$$
By Theorem~\req{thm1}, the first terms goes to zero. We are left to prove that the second term 
goes to zero as well. To this end, we proceed as follows:
\begin{equation}
\|(T_h-\tilde{T}_h)|_{V^k_h}\|_{\mathcal{L}(L^2(\Omega))}=\sup_{g_h\in V_h^k}\dfrac{\|T_hg_h-\tilde{T}_hg_h\|_0}{\|g_h\|_0}
=\sup_{g_h\in V_h^k}\dfrac{\|u_h^s-\tilde{u}_h^s\|_0}{\|g_h\|_0},
\end{equation}
where $u_h^s$ and $\tilde{u}_h^s$ denote, respectively, the solutions of the discrete 
source problems~\req{discreteSource} 
and~\req{discreteSource2} 
corresponding to $g_h$.\\
Let $\delta_h=\tilde{u}_h^s-u^s_h$. It holds\\
$$
\begin{array}{lll}
|\delta_h|_1^2  & \le\dfrac{1}{\alpha_*}\displaystyle{\sum_{P\in\mathcal{T}_h}}a_h^P(\delta_h,\delta_h) & 
(\text{stability of }  a_h^P(\cdot,\cdot))\\
& = \dfrac{1}{\alpha_*}\displaystyle{\sum_{P\in\mathcal{T}_h}}\tilde{S}^P\Big((I-\Pi^0_k)g_h,(I-\Pi^0_k)\delta_h\Big) 
& \text{(discrete source 
problems~\req{discreteSource} and~\req{discreteSource2})}\\
& \le \dfrac{\tilde{c}_1}{\alpha_*} \displaystyle{\sum_{P\in\mathcal{T}_h}} \|(I-\Pi^0_k)g_h\|_{0,P}  \|(I-\Pi^0_k)\delta_h\|_{0,P} 
& \text{(stability of $\tilde{S}^P(\cdot,\cdot)$)}\\
& \le C \, h \|g_h\|_0 \, |\delta_h|_1 & \text{(stability of $\Pi^0_k$ and projection error)}
\end{array}\medskip
$$
and hence
\begin{equation}
\label{eq:stimah1}
|\delta_h|_1 \le C \, h \, \|g_h\|_0.
\end{equation}
Taking into account the Poincar\'e inequality and estimate~\req{stimah1}, we obtain
$$
\|\delta_h\|_0\le C_p |\delta_h|_1\le C\,h\,\|g_h\|_0,
$$
with $C_p$ denoting the Poincar\'e constant. 
We conclude the proof observing that 
$$
\dfrac{\|\delta_h\|_0}{\|g_h\|_0}\le C\,h\quad forall\ g_h\in V_h^k,
$$
which gives the uniform convergence {\bf{(1)}} in~\req{uniform}.
\end{proof}

We end this section stating the convergence theorem for the second discrete approximation of problem~\req{eigPbm}.

\begin{theorem}
\label{thm:stab}
The family of operators $\tilde{T}_h$ associated with problem~\req{discreteSource2} 
converges uniformly to the operator $T$ associated with problem~\req{sourcePbm}, that is,
$$
\|T-\tilde{T}_h\|_{\mathcal{L}(L^2(\Omega))}\to 0\ for\  h\to 0.
$$
Let $\lambda$ be an eigenvalue of problem~\req{eigPbm}, with multiplicity $m$, and denote 
the corresponding eigenspace by $\mathcal{E}_{\lambda}$. Then
exactly $m$ discrete eigenvalues 
$\tilde{\lambda}_{1,h},\cdots,\tilde{\lambda}_{m,h}$, which are repeated according to their respective multiplicities, converge to 
$\tilde{\lambda}$. Moreover, let $\mathcal{E}_{\tilde{\lambda},h}$ be the direct sum of the eigenspaces corresponding 
to the eigenvalues $\tilde{\lambda}_{1,h},\cdots,\tilde{\lambda}_{m,h}$. 
Then, there exists a positive number $h_0$ such that for $h\le h_0$ 
the following inequalities are true:
$$
\begin{array}{l}
|\tilde{\lambda}-\tilde{\lambda}_{i,h}|\le Ch^{2t}\quad\forall i=1,\cdots,m,\\
\hat{\delta}(\mathcal{E}_{\tilde{\lambda}},\mathcal{E}_{\tilde{\lambda},h})\le C h^t,
\end{array}
$$
where the non-negative constant C is independent of $h,\  t = \min\{k, r\},$ being $k$ the order of the method and $r$ 
the regurality index of the eigenfunction, and 
$\hat{\delta}(\mathcal{E}_{\tilde{\lambda}},\mathcal{E}_{\tilde{\lambda},h})$ denotes the 
gap between $\mathcal{E}_{\tilde{\lambda}}$ and $\mathcal{E}_{\tilde{\lambda},h}$.
\end{theorem}

\begin{proof}
The uniform convergence of $\tilde{T}_h$ to $T$ is a direct consequence of Theorems~\ref{thm:nassif} and~\ref{thm:p1p2}.  
Finally, as for the analysis of the non stabilized formulation, eigenmodes convergence can be proved following the same arguments 
as in Theorem 4.2. and 4.3. in~\cite{eig-plates}.
\end{proof}

\section{Numerical tests}
\label{se:experiments}

In this section we present four numerical experiments to test the  performance of the virtual element methodfor the bi-dimensional case,
in particular we confirm the a priori bounds on the error of the eigenvalue approximation provided
by Theorem \ref{eq:thm1} and Theorem \ref{thm:stab}.
We  consider the error quantities:
\begin{equation}\label{eq:errqnts}
\epsilon_{h, \lambda} := |\lambda - \lambda_h|.
\end{equation}

We briefly sketch two possible constructions of the stabilizing bilinear forms  $S^P(\cdot, \cdot)$ and $\tilde{S}^P(\cdot, \cdot)$ in \eqref{eq:discreteforms} and \eqref{eq:discretebstab}, respectively. The first  choice  follows a  standard VEM technique (cf. \cite{vem, VemGuide}), the second one  is a  new \textit{diagonal recipe} for the stabilization introduced lately in \cite{dassi}.
Let us denote with $\bar{\mathbf{v}}_h$, $\bar{\mathbf{w}}_h \in {\RE}^{N_P}$
the vectors containing the values of the $N_P$  local degrees of freedom associated to $v_h, w_h \in 
V_h^k(P)$. Then, we set
\begin{itemize}
\item \textbf{scalar stabilization:} 
\begin{equation}
\label{eq:original}
S^P (v_h, w_h) =  \sigma_P \, \bar{\mathbf{v}}_h^T \bar{\mathbf{w}}_h
\qquad \text{and} \qquad 
\tilde{S}^P (v_h, w_h) := \tau_P \, h_P^2 \, \bar{\mathbf{v}}_h^T \bar{\mathbf{w}}_h
\end{equation}
where the stability parameters $\sigma_P$ and $\tau_P$ are two positive $h$-independent constants. When non clearly mentioned, in the numerical tests we choose $\sigma_P$ as the mean value of the eigenvalues of the matrix stemming from the consistency  term 
$a^P \left({\Pi}_{k}^{\nabla} \cdot, \, {\Pi}_{k}^{\nabla} \cdot \right)$ for the grad-grad form (see \eqref{eq:discreteforms}). In the same way we pick $\tau_P$ as the mean value of the eigenvalues of the matrix resulting from the term $\frac{1}{h_P^2}(\Pi^0_{k}\cdot, \, \Pi^0_{k} \cdot)_P$ for the mass matrix (see \eqref{eq:discretebstab}). 
\item \textbf{diagonal stabilization:} 
\begin{equation}
\label{eq:recipe}
S^P (v_h, w_h) =  \bar{\mathbf{v}}_h^T \, D_P \, \bar{\mathbf{w}}_h
\qquad \text{and} \qquad 
\tilde{S}^P (v_h, w_h) :=  \bar{\mathbf{v}}_h^T \, \widetilde{D}_P \, \bar{\mathbf{w}}_h.
\end{equation}
where $D_P$ and $\widetilde{D}_P$ are two diagonal matrices defined by \cite{dassi}
\[
(D_P)_{i,i} = \max \left\{ 1, \, a^P \left({\Pi}_{k}^{\nabla} \phi_i, \, {\Pi}_{k}^{\nabla} \phi_i \right)_P \right\} \qquad i=1, \dots, N_P
\]
\[
(\widetilde{D}_P)_{i,i} = \max \left\{ h^2_P, \, \left({\Pi}_{k}^{0} \phi_i, \, {\Pi}_{k}^{0} \phi_i \right)_P \right\} \qquad i=1, \dots, N_P
\]
where $\phi_i$ denotes the $i$-th basis function.
\end{itemize}
 Using standard scaling arguments, in accordance with Remark~\req{scale1} 
and Remark~\req{scale2}, we notice that both stabilizations yield the correct scale for $S^P (\cdot, \cdot)$  and $\tilde{S}^P (\cdot, \cdot)$.

\begin{test}
\label{test1}
In the first test we consider the standard eigenvalue problem with homogeneous Dirichlet boundary conditions.

Regarding the computational domain, in the test we take the square domain $\Omega= (0,1) ^2$, which is partitioned using the following sequences of polygonal meshes:
\begin{itemize}
\item $\{ \mathcal{V}_h\}_h$: sequence of Voronoi meshes with $h=1/8, 1/16, 1/32, 1/64$,
\item $\{ \mathcal{T}_h\}_h$: sequence of triangular meshes with $h=1/4, 1/8, 1/16, 1/32$,
\item $\{ \mathcal{Q}_h\}_h$: sequence of square meshes with $h=1/8, 1/16, 1/32, 1/64$,
\item $\{ \mathcal{W}_h\}_h$: sequence of WEB-like meshes with $h=1/8, 1/16, 1/32, 1/64$.
\end{itemize}
An example of the adopted meshes is shown in Figure \ref{Figure1}.  

\begin{figure}[!h]
\centering
\includegraphics[scale=0.66]{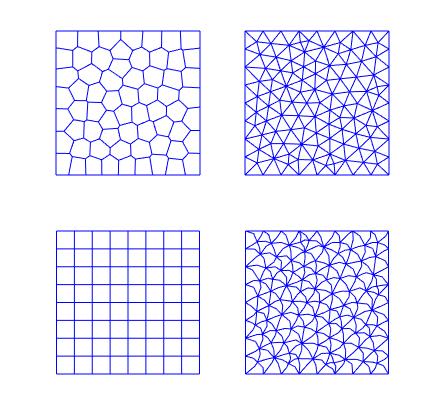} 
\caption{Test 7.1. Example of polygonal meshes: $\mathcal{V}_{1/8}$,  $\mathcal{T}_{1/8}$, $\mathcal{Q}_{1/8}$, $\mathcal{W}_{1/8}$.}
\label{Figure1}
\end{figure}

For the generation of the Voronoi meshes we used the code Polymesher \cite{TPPM12}. 
The WEB-like meshes are composed by hexagons, generated starting from the triangular meshes $\{ \mathcal{T}_h\}_h$ and randomly displacing the midpoint of each (non boundary) edge.\\
It is well known that the eigenvalues of the problem are given by
\[
\lambda = \pi^2 (n^2 + m^2) \qquad for\  n, m\in\mathbb{N},\  with\ n,m \neq 0.
\]
We consider the VEM approximation problem \eqref{eq:discreteEigPbm2} stemming from the stabilized bilinear form $\tilde{b}_h(\cdot, \cdot)$  where we use the stabilization \eqref{eq:original} with the above mentioned selections of the stabilization parameters. We consider the polynomial degree of accuracy  $k=1,2,3,4$ and we study the convergence of the errors $\epsilon_{h, \lambda}$ with respect to $h$ for the first six eigenvalues.

In Figures \ref{Figure2}-\ref{Figure5},  we display the results for the sequence of Voronoi    
meshes $\mathcal{V}_h$, the sequence of meshes $\mathcal{T}_h$, $\mathcal{Q}_h$, and 
$\mathcal{W}_h$, respecitvely.

\begin{figure}[!h]
\centering
\includegraphics[scale=0.2]{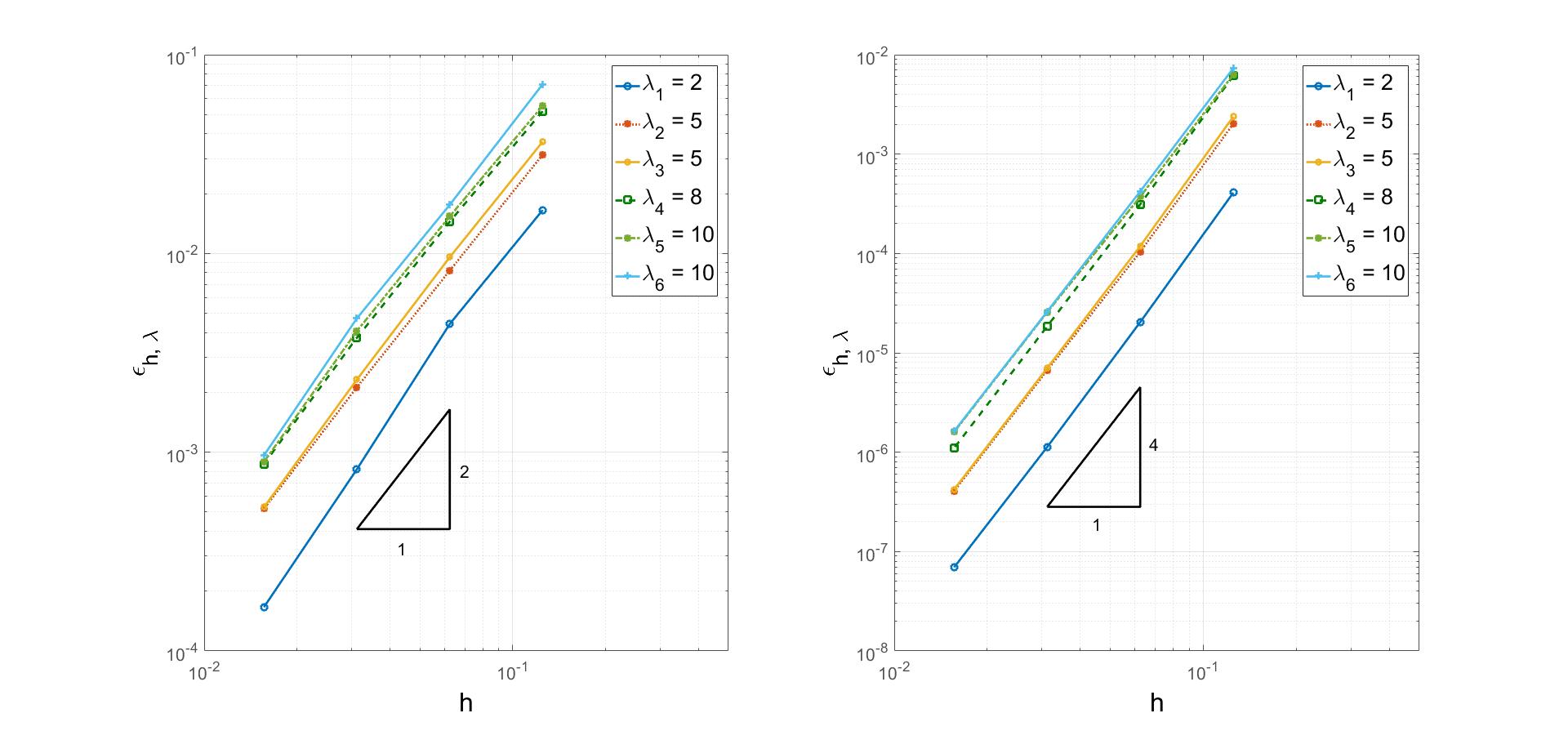} \\
\includegraphics[scale=0.2]{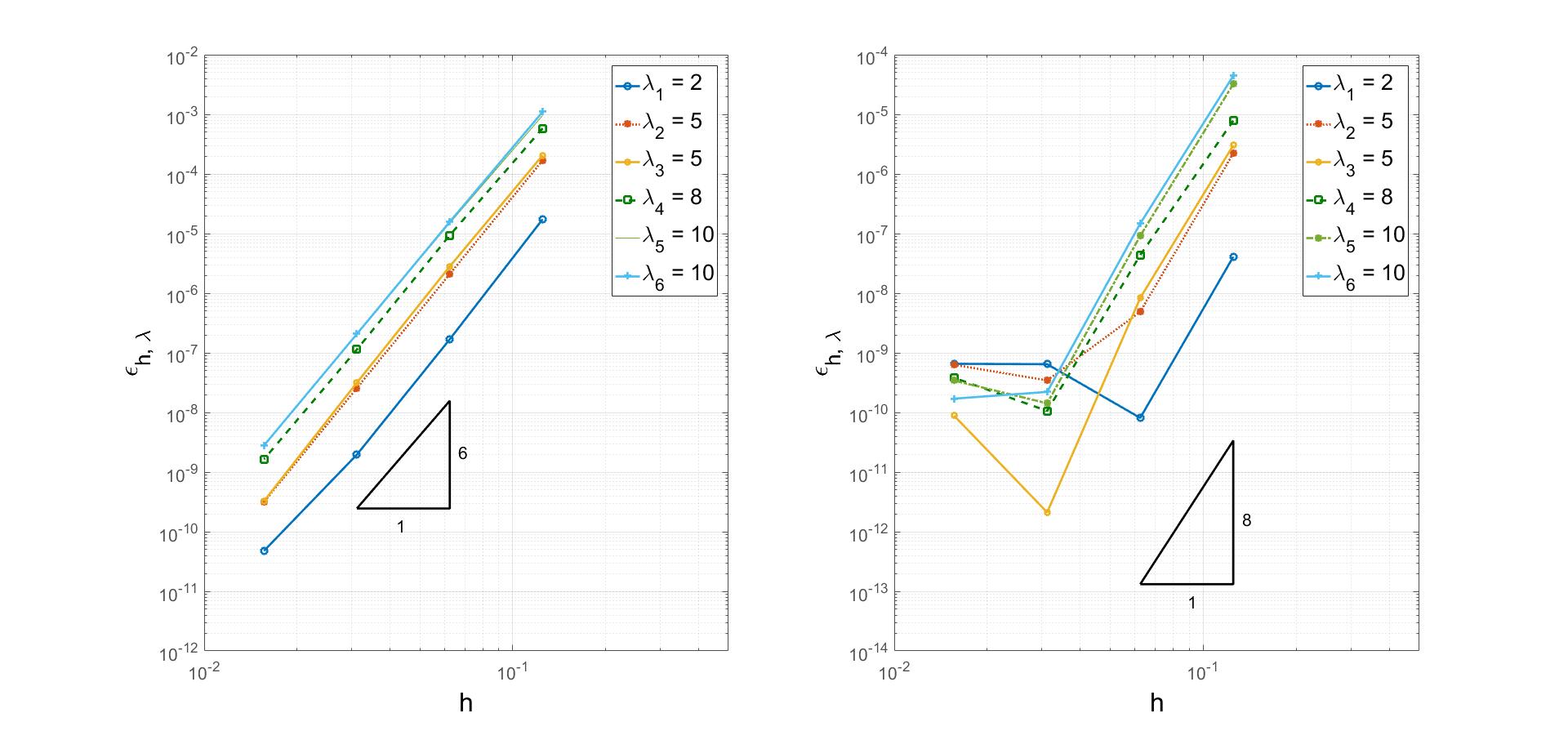} 
\caption{Test 7.1. Convergence plot for the eigenvalues for the sequence of meshes $\mathcal{V}_h$ with $k=1,2,3,4$.}
\label{Figure2}
\end{figure}

\begin{figure}[!h]
\centering
\includegraphics[scale=0.2]{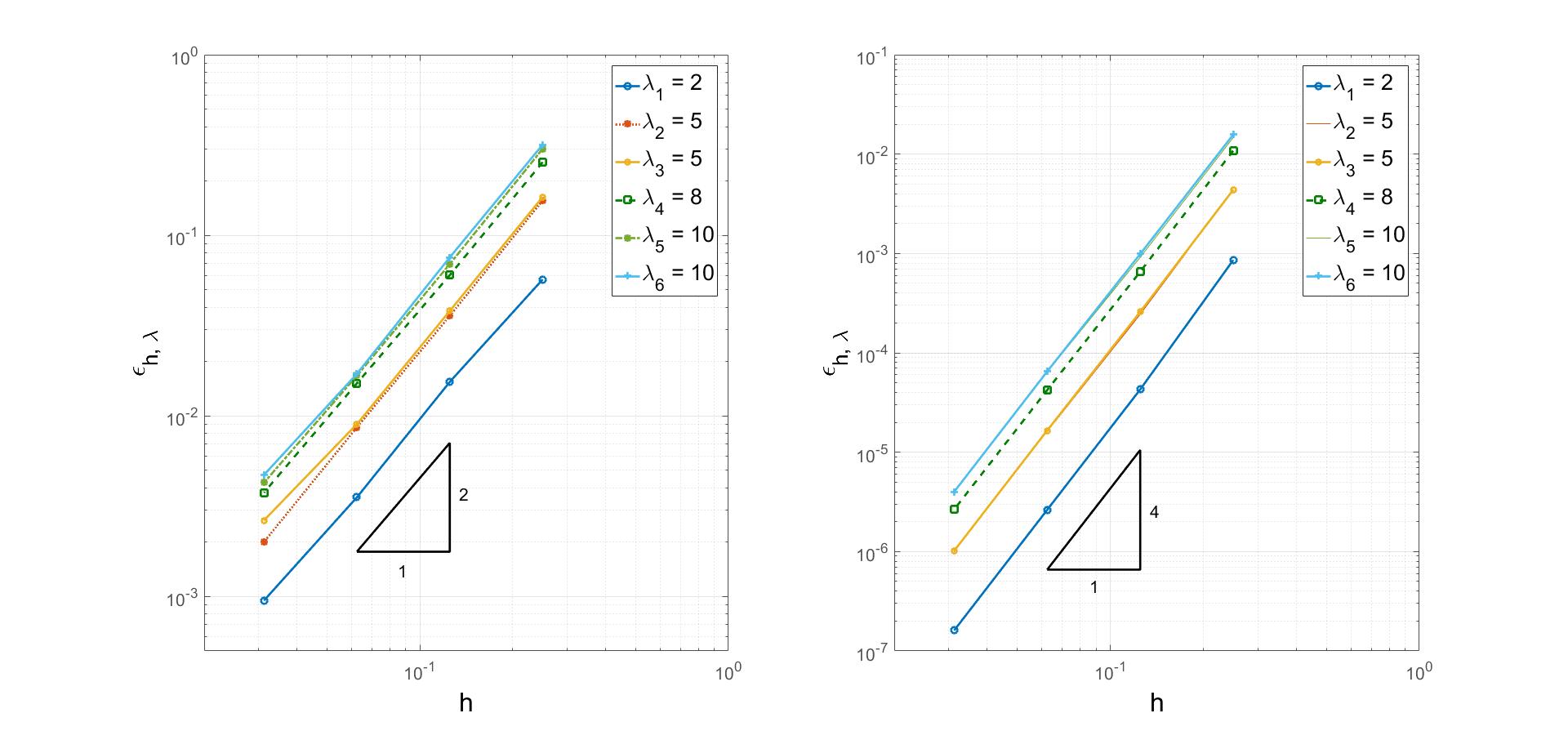} \\
\includegraphics[scale=0.2]{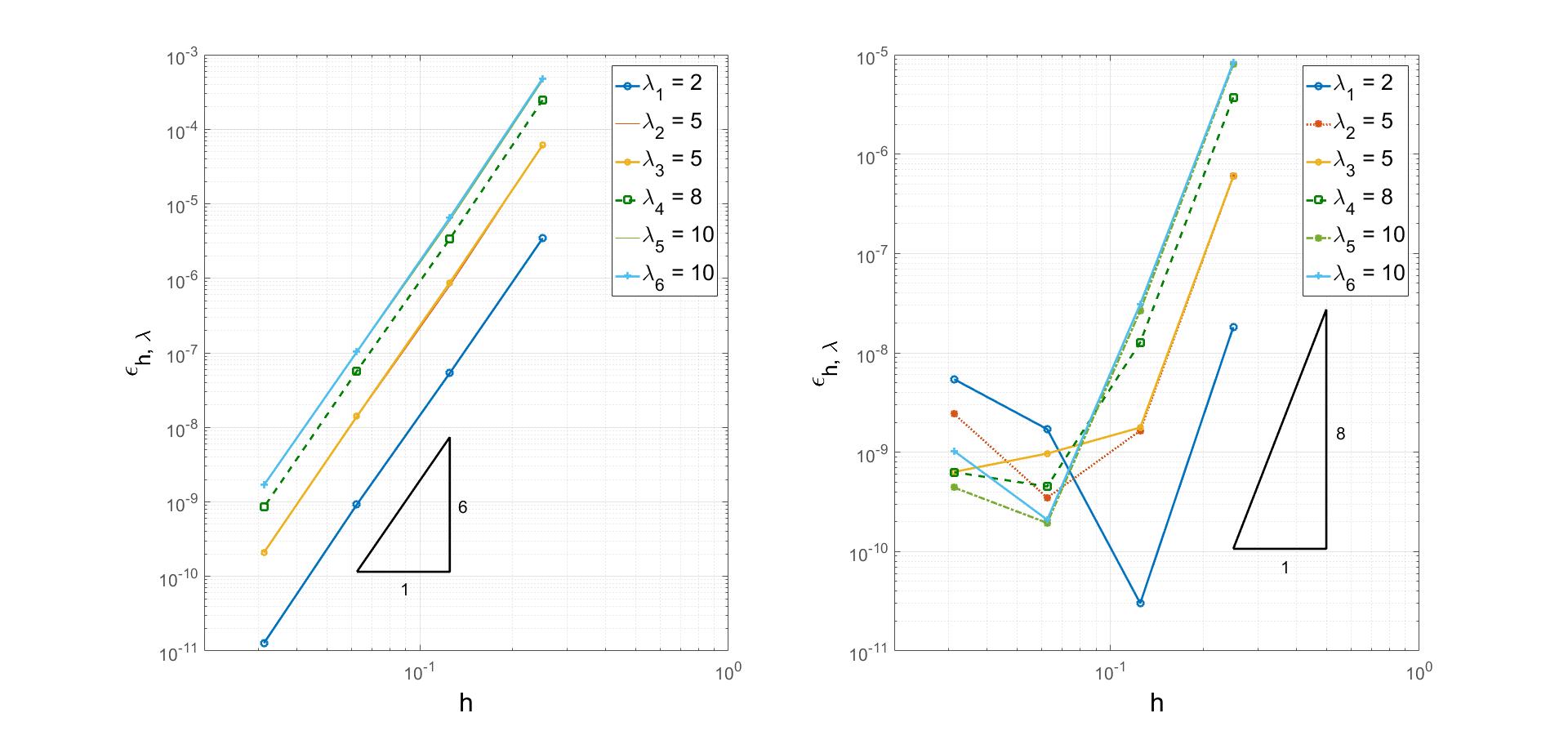}
\caption{Test 7.1. Convergence plot for the eigenvalues for the sequence of meshes $\mathcal{T}_h$ with $k=1,2,3,4$.}
\label{Figure3}
\end{figure}

\begin{figure}[!h]
\centering
\includegraphics[scale=0.2]{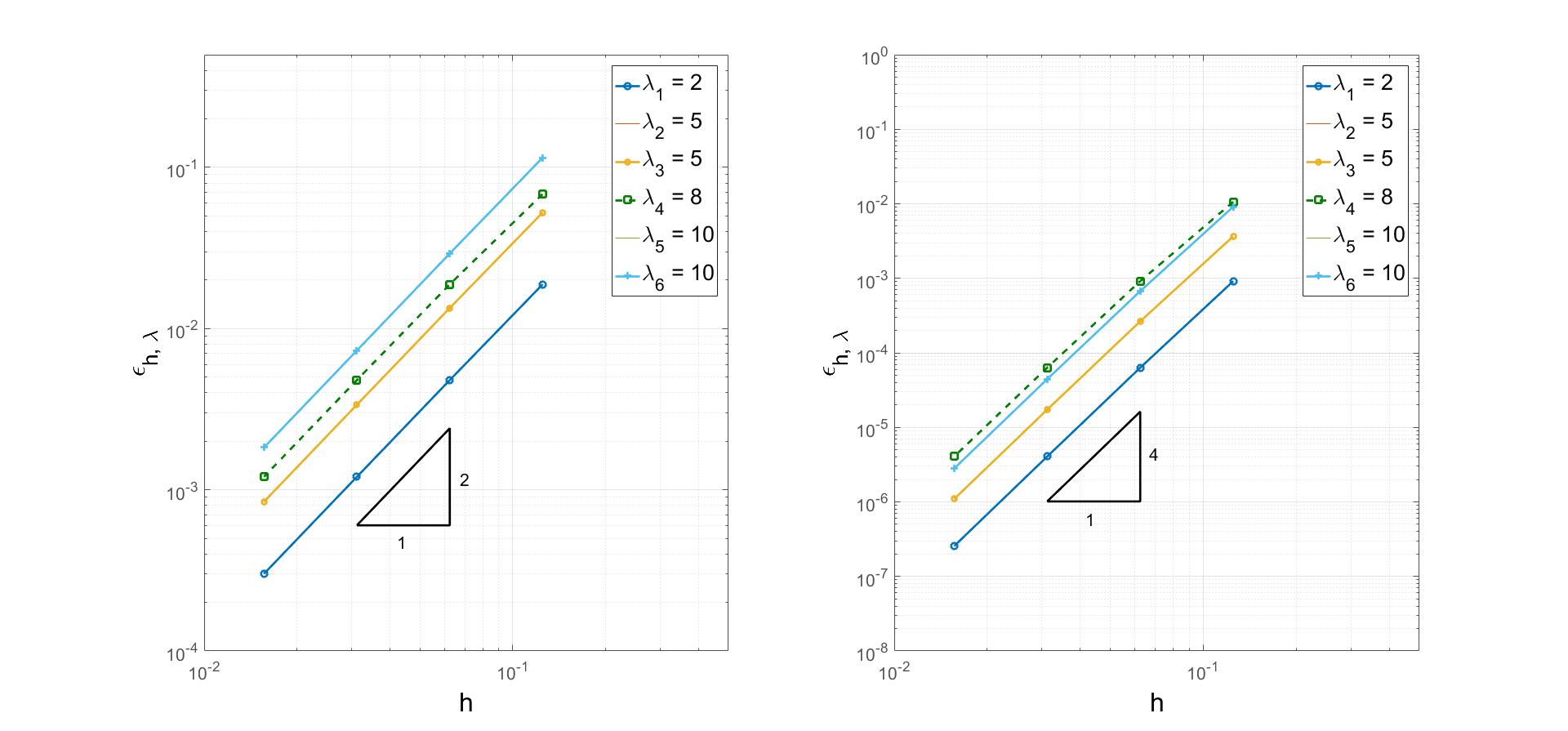} \\
\includegraphics[scale=0.2]{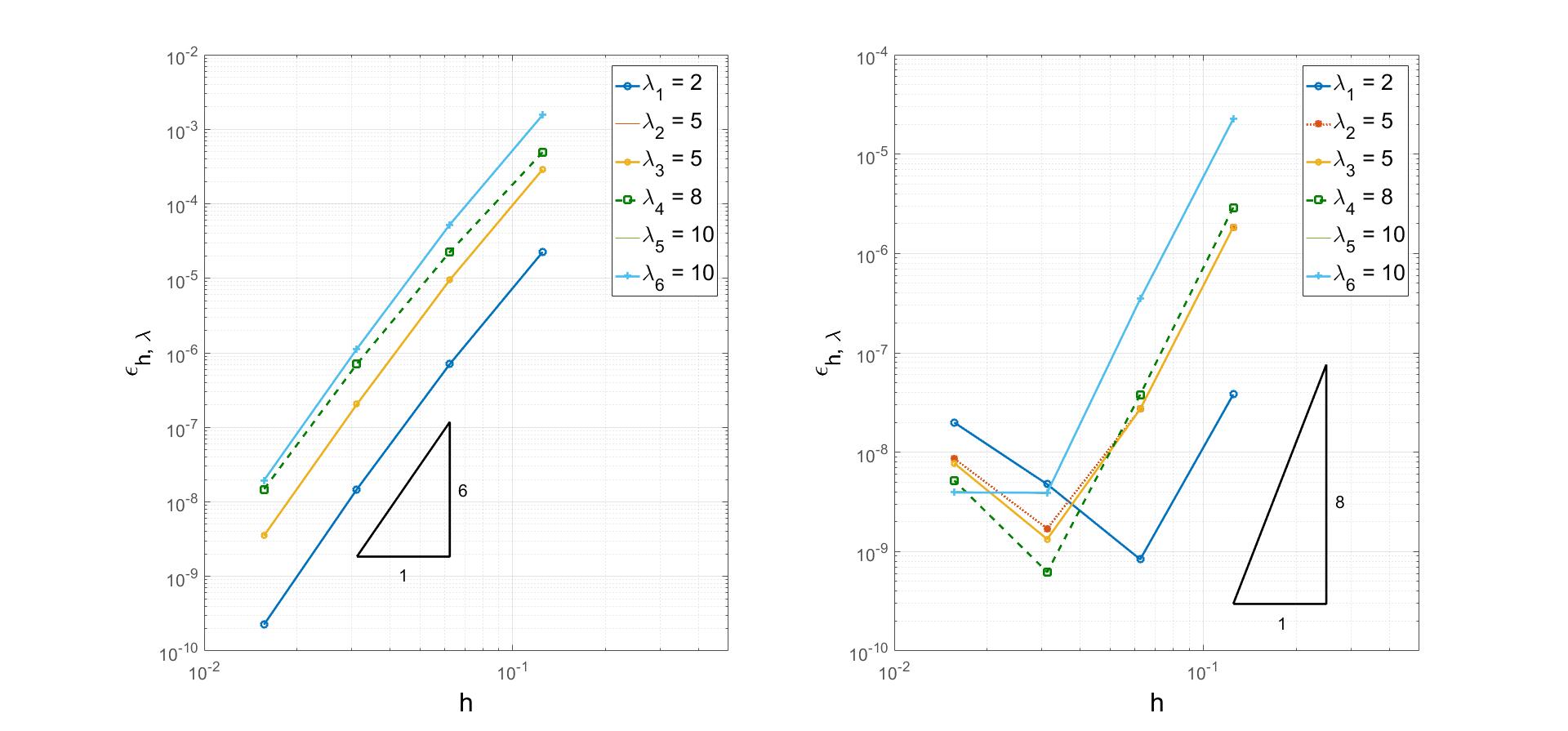}
    \caption{Test 7.1. Convergence plot for the eigenvalues for the sequence of meshes $\mathcal{Q}_h$ with $k=1,2,3,4$.}
\label{Figure4}
\end{figure}

\begin{figure}[!h]
\centering
\includegraphics[scale=0.2]{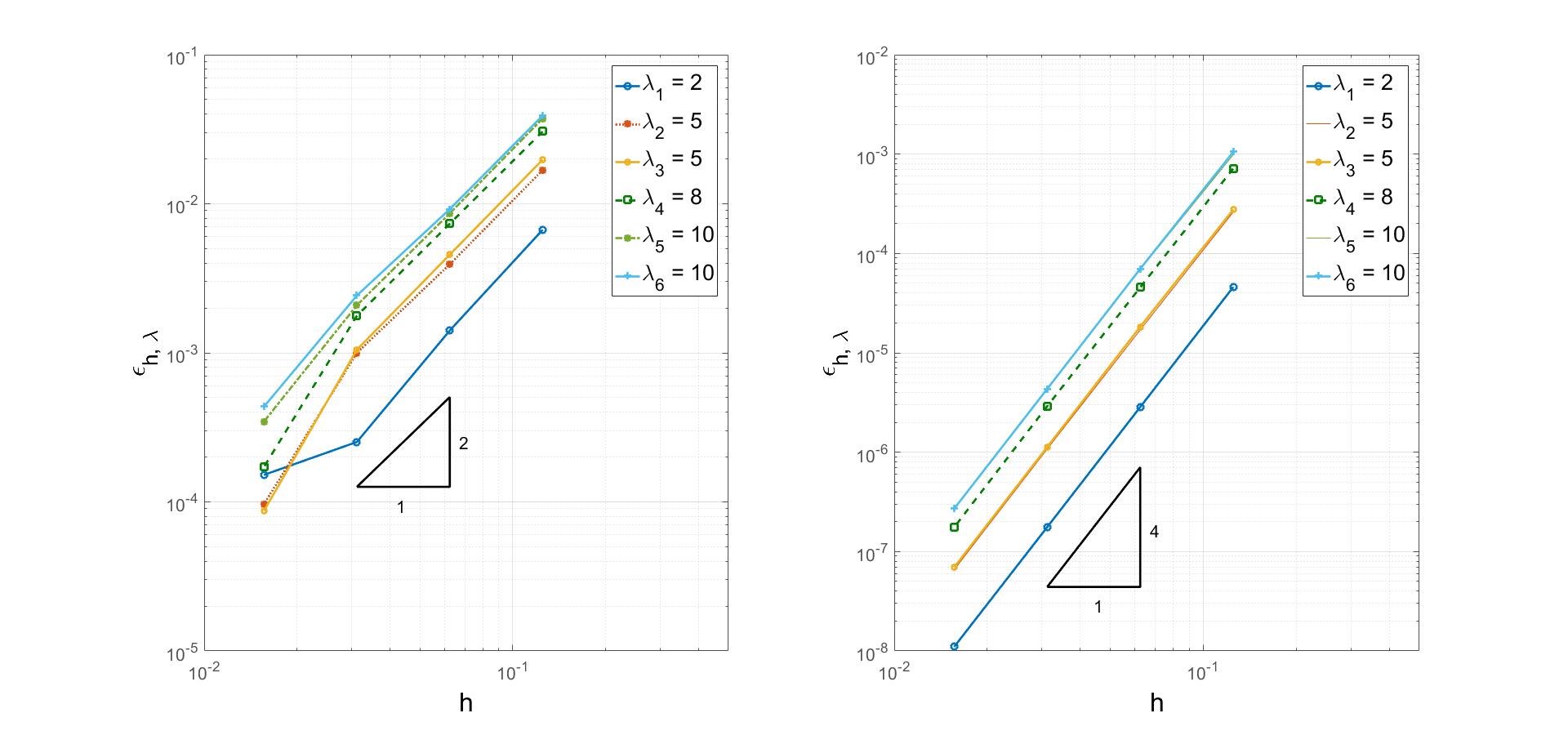} \\
\includegraphics[scale=0.2]{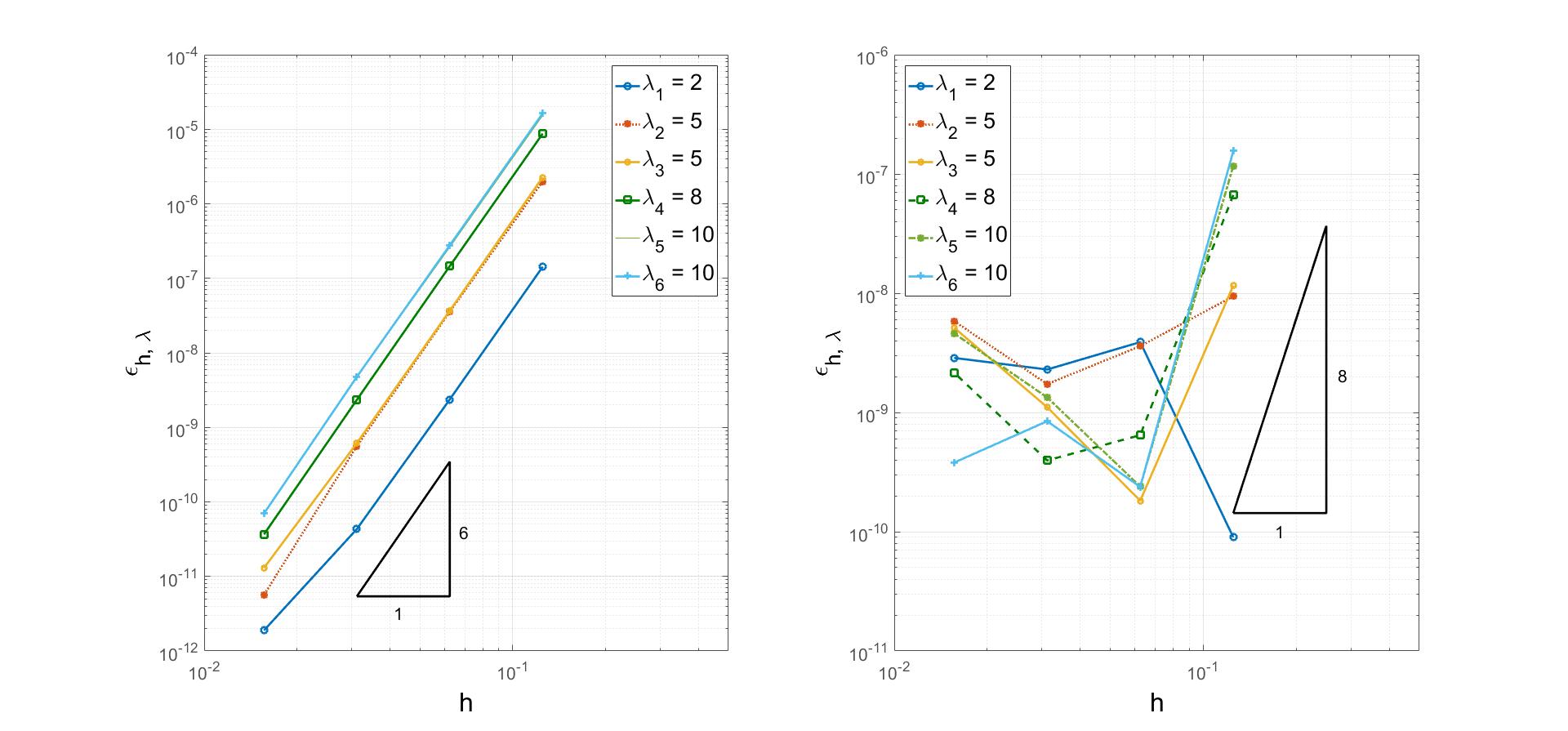}
\caption{Test 7.1. Convergence plot for the eigenvalues for the sequence of meshes $\mathcal{W}_h$ with $k=1,2,3,4$.}
\label{Figure5}
\end{figure}
We notice that for $k=1,2,3$ the theoretical predictions of Section \ref{sub:convstab} and Theorem \ref{thm:stab} are confirmed for all the adopted meshes (noticed that the eigenfunctions are analytical).
Whereas for $k=4$, the errors are close to the machine precision,
but for small values of $h$  they become larger. This fact is natural and stems from the conditioning of the matrices involved in the computation of the VEM solution. Indeed, as in standard FEM, their condition
numbers become larger when we consider higher VEM approximation degrees. 
The choice of the diagonal stabilization \eqref{eq:recipe}, cures this
problem, see Figure  \ref{Figure6}: for small values of the mesh size $h$ we have errors close to the machine precision.

\begin{figure}[!h]
\centering
\includegraphics[scale=0.2]{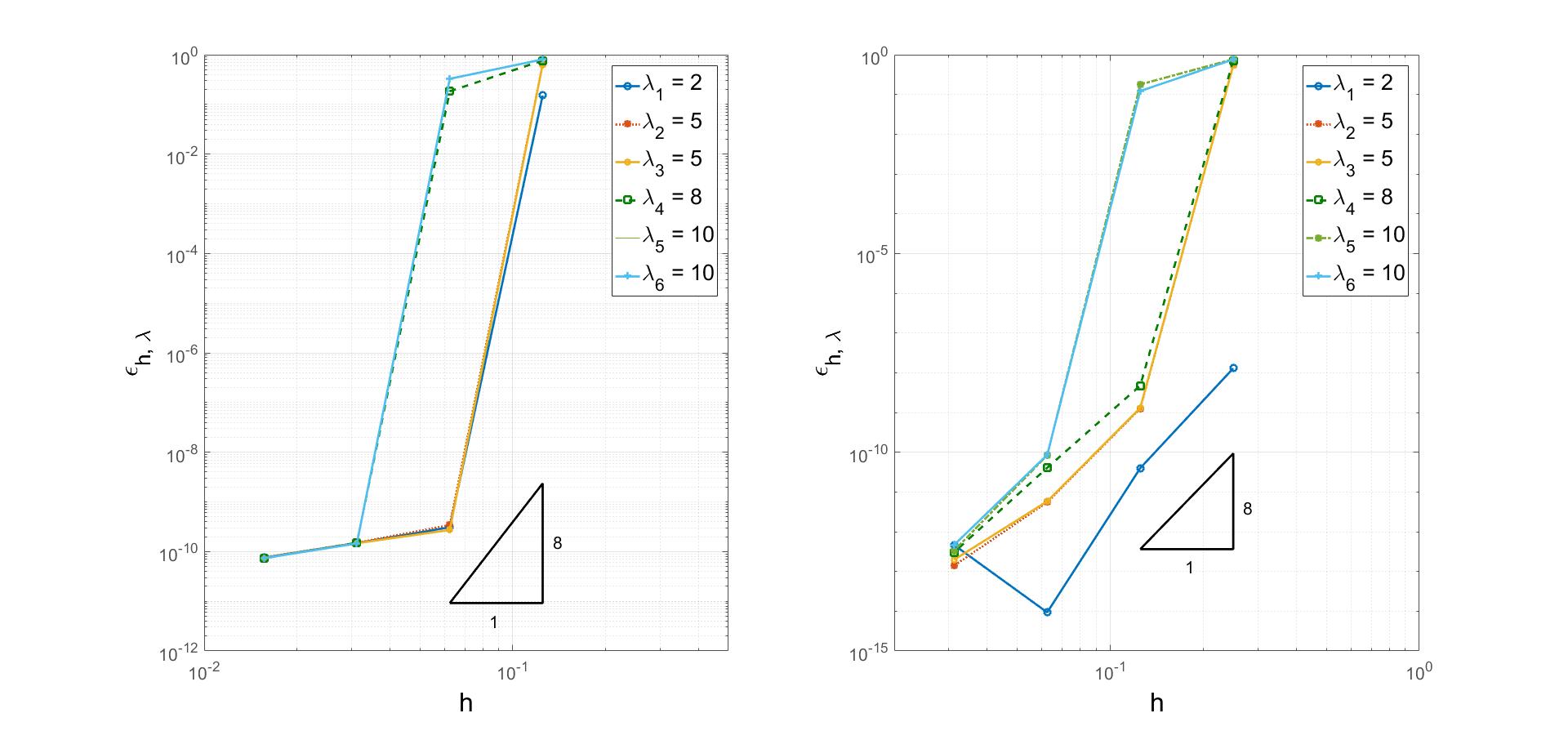} \\
\includegraphics[scale=0.2]{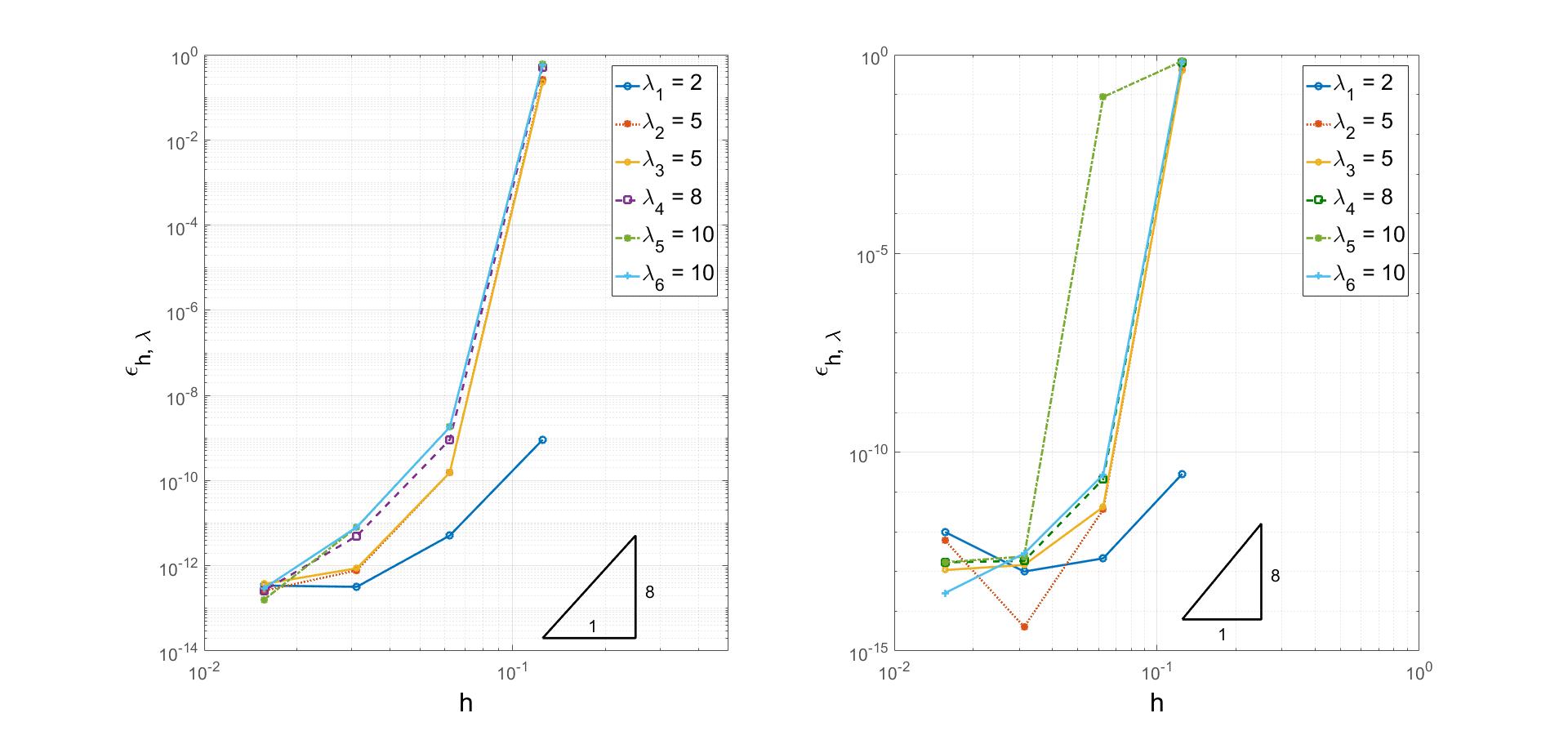}
\caption{Test 7.1. Convergence plots for the eigenvalues for the sequence of meshes $\mathcal{V}_h$, $\mathcal{T}_h$, $\mathcal{Q}_h$, $\mathcal{W}_h$ with $k=4$ and diagonal stabilization \eqref{eq:recipe}.}
\label{Figure6}
\end{figure}
\end{test}

\begin{test}
\label{test2}
We consider the same eigenvalue problem of Test \ref{test1} and we study the performance of the VEM discretization by comparing the non stabilized virtual method \eqref{eq:discreteEigPbm} with the stabilized one \eqref{eq:discreteEigPbm2} with the stabilizations above introduced (cf. \eqref{eq:original} and \eqref{eq:recipe}).
We use the polygonal decompositions listed above and polynomial degree of accuracy $k=1$ and $k=4$. In Table \ref{table1}, \ref{table2}, \ref{table3}, \ref{table4} we show respectively the results for the sequences of meshes $\mathcal{V}_h$, $\mathcal{T}_h$, $\mathcal{Q}_h$ and $\mathcal{W}_h$ for the lowest degree $k=1$. 

We notice that the matrices stemming from the non stabilized bilinear forms are not positive definite therefore we can not use the MATLAB routine \textbf{eigs} for sparse matrices. We overcome this problem by using for $k=1$ the MATLAB routine \textbf{eig} for full matrices. Whereas for $k=4$ we  approximate the non stabilized bilinear forms by considering the stabilizations $\tilde{S}^P(\cdot, \cdot)$ in \eqref{eq:original} with $\tau^P = 1e-16$.

Moreover, we observe that the stabilized method (with both stabilizations) and the non stabilized method exhibit almost identical errors for the low order $k=1$, at least for this example and with the adopted meshes. For $k=4$, as observed above, the diagonal stabilization shows for small values of $h$ better performances.

\begin{table}[!h]
\centering
\begin{tabular}{ll*{4}{c}}
\toprule
&    h          & $\lambda_1 = 2\pi^2$        & $\lambda_2 =  5\pi^2$      & $\lambda_4 = 8\pi^2$       & $\lambda_5 = 10\pi^2$       \\
\midrule
\multirow{4}*{scalar}
& $1/8$  & $1.65354e-2$ &  $3.65029e-2$ &  $5.19575e-2$ & $7.10714e-2$\\
& $1/16$ & $4.41849e-3$ &  $9.57783e-3$ &  $1.44610e-2$ & $1.75766e-2$ \\
& $1/32$ & $8.20606e-4$ &  $2.32037e-3$ &  $3.76823e-3$ & $4.71461e-3$ \\
& $1/64$ & $1.66238e-4$ &  $5.28919e-4$ &  $8.72219e-4$ & $9.64257e-4$ \\
\midrule
\multirow{4}*{diagonal}
& $1/8$  & $1.24793e-2$ &  $1.31338e-2$ &  $2.71426e-1$ & $3.95806e-1$\\
& $1/16$ & $4.37700e-3$ &  $8.84813e-3$ &  $1.17396e-2$ & $1.40754e-2$ \\
& $1/32$ & $8.64749e-4$ &  $2.39206e-3$ &  $3.79681e-3$ & $4.74322e-3$ \\
& $1/64$ & $1.80022e-4$ &  $5.61628e-4$ &  $9.21153e-4$ & $1.02400e-3$ \\
\midrule
\multirow{4}*{non stab}
& $1/8$  & $1.76616e-2$ &  $4.09826e-2$ &  $6.38367e-2$ & $8.48765e-2$\\
& $1/16$ & $4.61104e-3$ &  $1.02615e-2$ &  $1.55774e-2$ & $1.92623e-2$ \\
& $1/32$ & $8.78350e-4$ &  $2.46463e-3$ &  $4.00339e-3$ & $5.03841e-3$ \\
& $1/64$ & $1.80805e-4$ &  $5.66302e-4$ &  $9.34035e-4$ & $1.04229e-3$ \\
\bottomrule
\end{tabular}
\vspace{0.2cm}
\caption{Test 7.2. $\epsilon_{h,\lambda}$ for the sequence of meshes $\mathcal{V}_h$ with $k=1$ using stabilized form and non stabilized bilinear form (cf. \eqref{eq:discreteEigPbm2} and \eqref{eq:discreteEigPbm}).}
\label{table1}
\end{table}

\begin{table}[!h]
\centering
\begin{tabular}{ll*{4}{c}}
\toprule
&    h          & $\lambda_1 = 2\pi^2$        & $\lambda_2 =  5\pi^2$      & $\lambda_4 = 8\pi^2$       & $\lambda_5 = 10\pi^2$       \\
\midrule
\multirow{4}*{scalar}
& $1/4$  & $5.70169e-2$ &  $1.63047e-1$ &  $2.54353e-1$ & $3.17892e-1$\\
& $1/8$  & $1.54372e-2$ &  $3.81409e-2$ &  $6.02252e-2$ & $7.55284e-2$ \\
& $1/16$ & $3.54381e-3$ &  $9.00524e-3$ &  $1.51584e-2$ & $1.71163e-2$ \\
& $1/32$ & $9.51096e-4$ &  $2.64316e-3$ &  $3.76702e-3$ & $4.71650e-3$ \\
\midrule
\multirow{4}*{diagonal}
& $1/4$  & $5.70169e-2$ &  $1.63047e-1$ &  $2.54353e-1$ & $3.17892e-1$\\
& $1/8$  & $1.54372e-2$ &  $3.81409e-2$ &  $6.02252e-2$ & $7.55284e-2$ \\
& $1/16$ & $3.54381e-3$ &  $9.00524e-3$ &  $1.51584e-2$ & $1.71163e-2$ \\
& $1/32$ & $9.51096e-4$ &  $2.64316e-3$ &  $3.76702e-3$ & $4.71650e-3$ \\
\midrule
\multirow{4}*{non stab}
& $1/4$  & $5.70169e-2$ &  $1.63047e-1$ &  $2.54353e-1$ & $3.17892e-1$\\
& $1/8$  & $1.54372e-2$ &  $3.81409e-2$ &  $6.02252e-2$ & $7.55284e-2$ \\
& $1/16$ & $3.54381e-3$ &  $9.00524e-3$ &  $1.51584e-2$ & $1.71163e-2$ \\
& $1/32$ & $9.51096e-4$ &  $2.64316e-3$ &  $3.76702e-3$ & $4.71650e-3$ \\ 
\bottomrule
\end{tabular}
\vspace{0.2cm}
\caption{Test 7.2. $\epsilon_{h,\lambda}$ for the sequence of meshes $\mathcal{T}_h$ with $k=1$ using stabilized and non stabilized bilinear form (cf. \eqref{eq:discreteEigPbm2} and \eqref{eq:discreteEigPbm}).}
\label{table2}
\end{table}

\begin{table}[!h]
\centering
\begin{tabular}{ll*{4}{c}}
\toprule
&    h          & $\lambda_1 = 2\pi^2$        & $\lambda_2 =  5\pi^2$      & $\lambda_4 = 8\pi^2$       & $\lambda_5 = 10\pi^2$       \\
\midrule
\multirow{4}*{scalar}
& $1/8$  & $1.88180e-2$ &  $5.24872e-2$ &  $6.83463e-2$ & $1.14922e-1$\\
& $1/16$ & $4.79185e-3$ &  $1.34179e-2$ &  $1.88180e-2$ & $2.92254e-2$ \\
& $1/32$ & $1.20310e-3$ &  $3.36892e-3$ &  $4.79185e-3$ & $7.30994e-3$ \\
& $1/64$ & $3.01091e-4$ &  $8.43074e-4$ &  $1.20310e-3$ & $1.82733e-3$ \\
\midrule
\multirow{4}*{diagonal}
& $1/8$  & $1.34980e-2$ &  $3.00644e-2$ &  $1.31356e-1$ & $2.29494e-1$\\
& $1/16$ & $4.46838e-3$ &  $1.21017e-2$ &  $1.34980e-2$ & $2.61752e-2$ \\
& $1/32$ & $1.18304e-3$ &  $3.28834e-3$ &  $4.46838e-3$ & $7.12721e-3$ \\
& $1/64$ & $2.99840e-4$ &  $8.38065e-4$ &  $1.18304e-3$ & $1.81604e-3$ \\
\midrule
\multirow{4}*{non stab}
& $1/8$  & $1.96813e-2$ &  $5.61954e-2$ &  $8.40240e-2$ & $1.24337e-1$\\
& $1/16$ & $4.84403e-3$ &  $1.36304e-2$ &  $1.96812e-2$ & $2.97190e-2$ \\
& $1/32$ & $1.20633e-3$ &  $3.38191e-3$ &  $4.84400e-3$ & $7.33940e-3$ \\
& $1/64$ & $3.01292e-4$ &  $8.43879e-4$ &  $1.20632e-3$ & $1.82915e-3$ \\
\bottomrule
\end{tabular}
\vspace{0.2cm}
\caption{Test 7.2. $\epsilon_{h,\lambda}$ for the sequence of meshes $\mathcal{Q}_h$ with $k=1$ using stabilized and non stabilized bilinear form (cf. \eqref{eq:discreteEigPbm2} and \eqref{eq:discreteEigPbm}).}
\label{table3}
\end{table}

\begin{table}[!h]
\centering
\begin{tabular}{ll*{4}{c}}
\toprule
&    h          & $\lambda_1 = 2\pi^2$        & $\lambda_2 =  5\pi^2$      & $\lambda_4 = 8\pi^2$       & $\lambda_5 = 10\pi^2$       \\
\midrule
\multirow{4}*{scalar}
& $1/8$  & $6.69124e-3$ &  $1.97714e-2$ &  $3.05479e-2$ & $3.91016e-2$\\
& $1/16$ & $1.41184e-3$ &  $4.55875e-3$ &  $7.34674e-3$ & $9.20778e-3$ \\
& $1/32$ & $2.52359e-4$ &  $1.04890e-3$ &  $1.77164e-3$ & $2.42393e-3$ \\
& $1/64$ & $1.51786e-4$ &  $8.71580e-5$ &  $1.71693e-4$ & $4.35676e-4$ \\
\midrule
\multirow{4}*{diagonal}
& $1/8$  & $5.41188e-3$ &  $1.40607e-2$ &  $1.70178e-2$ & $1.73017e-2$\\
& $1/16$ & $1.19325e-3$ &  $3.91139e-3$ &  $6.05178e-3$ & $7.56699e-3$ \\
& $1/32$ & $2.07297e-4$ &  $9.27293e-4$ &  $1.56408e-3$ & $2.16149e-3$ \\
& $1/64$ & $1.62207e-4$ &  $6.16087e-5$ &  $1.28674e-4$ & $3.81062e-4$ \\
\midrule
\multirow{4}*{non stab}
& $1/8$  & $6.06830e-3$ &  $1.83072e-2$ &  $2.85770e-2$ & $3.71010e-2$\\
& $1/16$ & $1.23021e-3$ &  $4.12929e-3$ &  $6.68713e-3$ & $8.43576e-3$ \\
& $1/32$ & $2.09025e-4$ &  $9.37654e-3$ &  $1.59110e-3$ & $2.20261e-3$ \\
& $1/64$ & $1.62237e-4$ &  $6.14193e-5$ &  $1.28187e-4$ & $3.80246e-4$ \\
\bottomrule
\end{tabular}
\vspace{0.2cm}
\caption{Test 7.2. $\epsilon_{h,\lambda}$ for the sequence of meshes $\mathcal{W}_h$ with $k=1$ using stabilized and non stabilized bilinear form (cf. \eqref{eq:discreteEigPbm2} and \eqref{eq:discreteEigPbm}).}
\label{table4}
\end{table}

\begin{table}[!h]
\centering
\begin{tabular}{ll*{4}{c}}
\toprule
&    h          & $\lambda_1 = 2\pi^2$        & $\lambda_2 =  5\pi^2$      & $\lambda_4 = 8\pi^2$       & $\lambda_5 = 10\pi^2$       \\
\midrule
\multirow{4}*{scalar}
& $1/8$  & $4.08694e-08$ &  $3.03255e-06$ &  $7.85639e-06$ & $4.42827e-05$\\
& $1/16$ & $8.10132e-11$ &  $8.32945e-09$ &  $4.36448e-08$ & $1.47462e-07$ \\
& $1/32$ & $6.45147e-10$ &  $2.10939e-12$ &  $1.06202e-10$ & $2.21067e-10$ \\
& $1/64$ & $6.54068e-10$ &  $8.81773e-11$ &  $3.81458e-10$ & $1.70565e-10$ \\
\midrule
\multirow{4}*{diagonal}
& $1/8$  & $1.50768e-01$ &  $6.34515e-01$ &  $7.65639e-01$ & $8.05257e-01$\\
& $1/16$ & $3.09080e-10$ &  $2.72011e-10$ &  $1.83621e-01$ & $3.27562e-01$ \\
& $1/32$ & $1.48304e-10$ &  $1.47156e-10$ &  $1.50437e-10$ & $1.45682e-10$ \\
& $1/64$ & $7.55412e-11$ &  $7.34704e-11$ &  $7.29845e-11$ & $7.36543e-11$ \\
\midrule
\multirow{4}*{non stab}
& $1/8$  & $4.09387e-08$ &  $3.04614e-06$ &  $7.91452e-06$ & $4.46984e-05$\\
& $1/16$ & $8.09754e-11$ &  $8.34021e-09$ &  $4.37285e-08$ & $1.47462e-07$ \\
& $1/32$ & $6.44920e-10$ &  $1.28349e-12$ &  $1.05962e-10$ & $2.22278e-10$ \\
& $1/64$ & $6.54187e-10$ &  $8.82860e-11$ &  $3.81741e-10$ & $1.70739e-10$ \\
\bottomrule
\end{tabular}
\vspace{0.2cm}
\caption{Test 7.2. $\epsilon_{h,\lambda}$ for the sequence of meshes $\mathcal{V}_h$ with $k=4$ using stabilized and non stabilized bilinear form (cf. \eqref{eq:discreteEigPbm2} and \eqref{eq:discreteEigPbm}).}
\label{table5}
\end{table}

\begin{table}[!h]
\centering
\begin{tabular}{ll*{4}{c}}
\toprule
&    h          & $\lambda_1 = 2\pi^2$        & $\lambda_2 =  5\pi^2$      & $\lambda_4 = 8\pi^2$       & $\lambda_5 = 10\pi^2$       \\
\midrule
\multirow{4}*{scalar}
& $1/4$  & $1.81535e-08$ &  $5.97563e-07$ &  $3.71995e-06$ & $8.31271e-06$\\
& $1/8$  & $3.01783e-11$ &  $1.77195e-09$ &  $1.27304e-08$ & $3.09954e-08$ \\
& $1/16$ & $1.70546e-09$ &  $9.69266e-10$ &  $4.52175e-10$ & $2.09013e-10$ \\
& $1/32$ & $5.42087e-09$ &  $6.33818e-10$ &  $6.29748e-10$ & $1.02642e-09$ \\
\midrule
\multirow{4}*{diagonal}
& $1/4$  & $1.29083e-08$ &  $5.43812e-01$ &  $7.14882e-01$ & $7.59974e-01$\\
& $1/8$  & $3.97711e-11$ &  $1.30016e-09$ &  $4.74659e-09$ & $1.22090e-01$ \\
& $1/16$ & $9.35909e-15$ &  $5.81516e-12$ &  $4.05619e-11$ & $8.57591e-11$ \\
& $1/32$ & $4.62015e-13$ &  $1.97548e-13$ &  $3.04890e-13$ & $4.64787e-13$ \\
\midrule
\multirow{4}*{non stab}
& $1/4$  & $1.81674e-08$ &  $5.97558e-07$ &  $3.71998e-06$ & $8.31282e-06$\\
& $1/8$  & $3.16576e-11$ &  $1.77308e-09$ &  $1.27284e-08$ & $3.09934e-08$ \\
& $1/16$ & $1.70289e-09$ &  $9.67242e-10$ &  $4.51765e-10$ & $2.07578e-10$ \\
& $1/32$ & $5.41948e-09$ &  $6.33574e-10$ &  $6.30034e-10$ & $1.02564e-09$ \\
\bottomrule
\end{tabular}
\vspace{0.2cm}
\caption{Test 7.2. $\epsilon_{h,\lambda}$ for the sequence of meshes $\mathcal{T}_h$ with $k=4$ using stabilized and non stabilized bilinear form (cf. \eqref{eq:discreteEigPbm2} and \eqref{eq:discreteEigPbm}).}
\label{table6}
\end{table}

\begin{table}[!h]
\centering
\begin{tabular}{ll*{4}{c}}
\toprule
&    h          & $\lambda_1 = 2\pi^2$        & $\lambda_2 =  5\pi^2$      & $\lambda_4 = 8\pi^2$       & $\lambda_5 = 10\pi^2$       \\
\midrule
\multirow{4}*{scalar}
& $1/8$  & $3.81184e-08$ &  $1.84365e-06$ &  $2.90250e-06$ & $2.25968e-05$\\
& $1/16$ & $8.34788e-10$ &  $2.72448e-08$ &  $3.80267e-08$ & $3.49035e-07$ \\
& $1/32$ & $4.75842e-09$ &  $1.33322e-09$ &  $6.22220e-10$ & $3.91731e-09$ \\
& $1/64$ & $1.99260e-08$ &  $7.75128e-09$ &  $5.20401e-09$ & $3.93187e-09$ \\
\midrule
\multirow{4}*{diagonal}
& $1/8$  & $9.12570e-10$ &  $2.24086e-01$ &  $5.15054e-01$ & $5.83981e-01$\\
& $1/16$ & $5.13094e-12$ &  $1.54985e-10$ &  $9.12586e-10$ & $1.80273e-09$ \\
& $1/32$ & $3.20728e-13$ &  $8.71979e-13$ &  $5.02799e-12$ & $8.12758e-12$ \\
& $1/64$ & $3.42866e-13$ &  $3.75371e-13$ &  $2.60074e-13$ & $2.95603e-13$ \\
\midrule
\multirow{4}*{non stab}
& $1/8$  & $3.81202e-08$ &  $1.84390e-06$ &  $2.90262e-06$ & $2.26029e-05$\\
& $1/16$ & $8.34530e-10$ &  $2.72483e-08$ &  $3.80285e-08$ & $3.49125e-07$ \\
& $1/32$ & $4.75887e-09$ &  $1.33338e-09$ &  $6.22323e-10$ & $3.91791e-09$ \\
& $1/64$ & $1.99268e-08$ &  $7.75217e-09$ &  $5.20491e-09$ & $3.93284e-09$ \\
\bottomrule
\end{tabular}
\vspace{0.2cm}
\caption{Test 7.2. $\epsilon_{h,\lambda}$ for the sequence of meshes $\mathcal{Q}_h$ with $k=4$ using stabilized and non stabilized bilinear form (cf. \eqref{eq:discreteEigPbm2} and \eqref{eq:discreteEigPbm}).}
\label{table7}
\end{table}

\begin{table}[!h]
\centering
\begin{tabular}{ll*{4}{c}}
\toprule
&    h          & $\lambda_1 = 2\pi^2$        & $\lambda_2 =  5\pi^2$      & $\lambda_4 = 8\pi^2$       & $\lambda_5 = 10\pi^2$       \\
\midrule
\multirow{4}*{scalar}
& $1/8$  & $8.97772e-11$ &  $1.15946e-08$ &  $6.67721e-08$ & $1.55547e-07$\\
& $1/16$ & $3.88976e-09$ &  $1.81961e-10$ &  $6.45283e-10$ & $2.37503e-10$ \\
& $1/32$ & $2.29197e-09$ &  $1.10633e-09$ &  $3.96093e-10$ & $8.44730e-10$ \\
& $1/64$ & $2.86708e-09$ &  $5.13913e-09$ &  $2.15705e-09$ & $3.80039e-10$ \\
\midrule
\multirow{4}*{diagonal}
& $1/8$  & $2.85391e-11$ &  $4.01092e-01$ &  $6.25071e-01$ & $6.71266e-01$\\
& $1/16$ & $2.13639e-13$ &  $4.20180e-12$ &  $2.11403e-11$ & $2.73006e-11$ \\
& $1/32$ & $9.95303e-14$ &  $1.45713e-13$ &  $1.83942e-13$ & $2.82788e-13$ \\
& $1/64$ & $9.81984e-13$ &  $1.08277e-13$ &  $1.69903e-13$ & $2.86532e-14$ \\
\midrule
\multirow{4}*{non stab}
& $1/8$  & $9.63542e-11$ &  $1.16189e-08$ &  $6.70782e-08$ & $1.56541e-07$\\
& $1/16$ & $3.85156e-09$ &  $2.89697e-10$ &  $5.72624e-10$ & $2.81162e-10$ \\
& $1/32$ & $2.31192e-09$ &  $1.08736e-09$ &  $3.66330e-10$ & $8.35562e-10$ \\
& $1/64$ & $2.86720e-09$ &  $5.13207e-09$ &  $2.14902e-09$ & $3.75120e-10$ \\
\bottomrule
\end{tabular}
\vspace{0.2cm}
\caption{Test 7.2. $\epsilon_{h,\lambda}$ for the sequence of meshes $\mathcal{W}_h$ with $k=4$ using stabilized and non stabilized bilinear form (cf. \eqref{eq:discreteEigPbm2} and \eqref{eq:discreteEigPbm}).}
\label{table8}
\end{table}

We can observe that the results in the tables confirm the theoretical rates of convergence stated in Section \ref{sub:convnonstab} and \ref{sub:convstab}. 
 In Table \ref{table2} we observe that the errors, as expected, are identical for the three cases. Indeed for the triangular meshes $\mathcal{T}_h$ the virtual space $V_h^1(T)$ corresponds to the space of linear polynomials, then
\[
S^T\Big((I-\Pi_1^{\nabla})u_h,(I-\Pi_1^{\nabla})v_h\Big) = \tilde{S}^T\Big((I-\Pi_1^{0})u_h,(I-\Pi_1^{0})v_h\Big) = 0 
\]
for all $u_h, v_h \in V_h^1(T)$, therefore the non stabilized method \eqref{eq:discreteEigPbm} and the stabilized method \eqref{eq:discreteEigPbm2} are equivalent.

Finally, we test the robustness of the method with respect to the stabilization parameter $\tau_P$  in \eqref{eq:original}. In Figure \ref{parameter} we plot the first four eigenvalues obtained by using the method \eqref{eq:discreteEigPbm2} with $k=1$ for the sequence of Voronoi meshes $\mathcal{V}_h$ in Test \ref{test1} as a function of stabilization parameter $\tau_P$.

\begin{figure}[!h]
\centering
\includegraphics[scale=0.2]{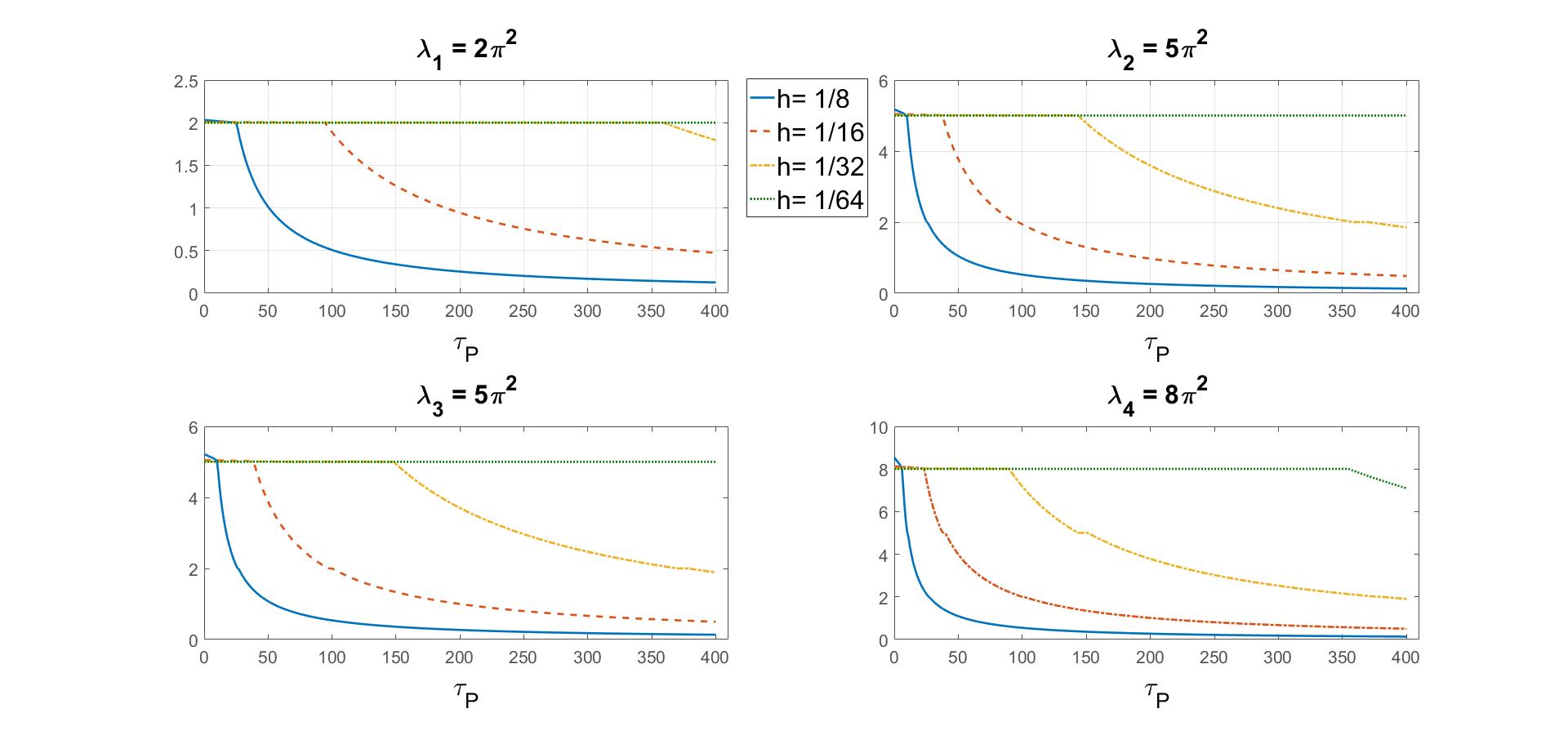} 
\caption{Test 7.2. First four eigenvalues as a function of the stabilization parameter $\tau_P$ (on the abscissas axis).}
\label{parameter}
\end{figure}

We observe that the method is robust with respect to the stabilization parameter $\tau_P$. For reasonable values of $\tau_P$ and for small enough values of the mesh size $h$, the numerical eigenvalues are not effected by the selection of the stabilization parameter.  Moreover, as expected, the ``critical parameter'',  i.e. the minimum value $\tau_P$ for which the associated method fails, goes like $h^{-2}$.
\end{test}

\begin{test}
\label{test3}
This test problem, as the following one, is taken from the benchmark singular solution set in \cite{dauge}.
We consider the square domain $\Omega = (-1, 1)^2$ splitted into two subdomains $\Omega_{\delta}$ and $\Omega_1$ (see the left plot in Figure \ref{split_domain}), and we study the eigenvalue problem on the square with Neumann homogeneous boundary conditions and discontinuous diffusivity. 
 \begin{figure}[!h]
\centering
\includegraphics[scale=0.8]{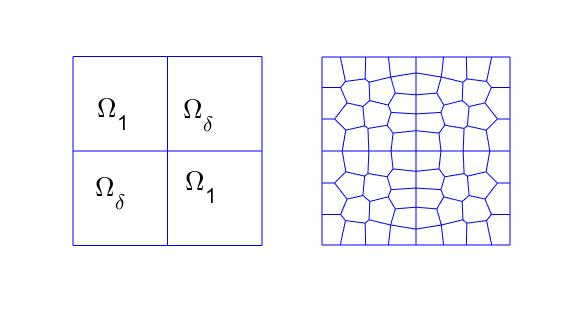}
\caption{Test 7.3. Left plot: subdivision of $\Omega$ into the subdomains $\Omega_{\delta}$ and $\Omega_1$. Right plot: Example of locally Voronoi decomposition of $\Omega$.}
\label{split_domain}
\end{figure}
In this test we consider the continuous bilinear form
\[
a_{\K}^P(u, v):= \int_P \K \nabla u \cdot \nabla v \, {\rm d}\bf{x}
\]
whose virtual approximation (see \cite{vemgeneral}) is given by
\begin{equation}\label{eq:ahgeneral}
a_{h, \K}^P(u_h,v_h)=  \int_P \K \Pi_{k-1}^0 \nabla u_h \cdot \Pi_{k-1}^0 \nabla v_h \, {\rm d}{\bf{x}}
+ S^P\Big((I-\Pi_k^{\nabla})u_h,(I-\Pi_k^{\nabla})v_h\Big)
\end{equation}
to be used in place of $a_{h}^P(u_h,v_h)$ (cf. \eqref{eq:discreteforms}) in Problem \eqref{eq:discreteSource2}.
We take $\K_{|\Omega_1} = I$ and $\K_{|\Omega_{\delta}} = \delta I$ with four different
values of $\delta$, namely $\delta = 0.50, 0.10, 0.01, 1e-8$.

We apply the Virtual Element method \eqref{eq:discreteEigPbm2} with the scalar stabilization \eqref{eq:original} using a sequence of Voronoi meshes with mesh diameter $h=1/2, 1/4, 1/8, 1/16$ (see the right plot in Figure \ref{split_domain} for an example of the adopted meshes). We show the plot of the convergence for the first and second computed eigenvalues in Figures \ref{test3_first} and \ref{test3_second}. We compute the errors $\epsilon_{h, \lambda}$ by comparing our results with the values given in \cite{dauge}.

\begin{figure}[!h]
\centering
\includegraphics[scale=0.2]{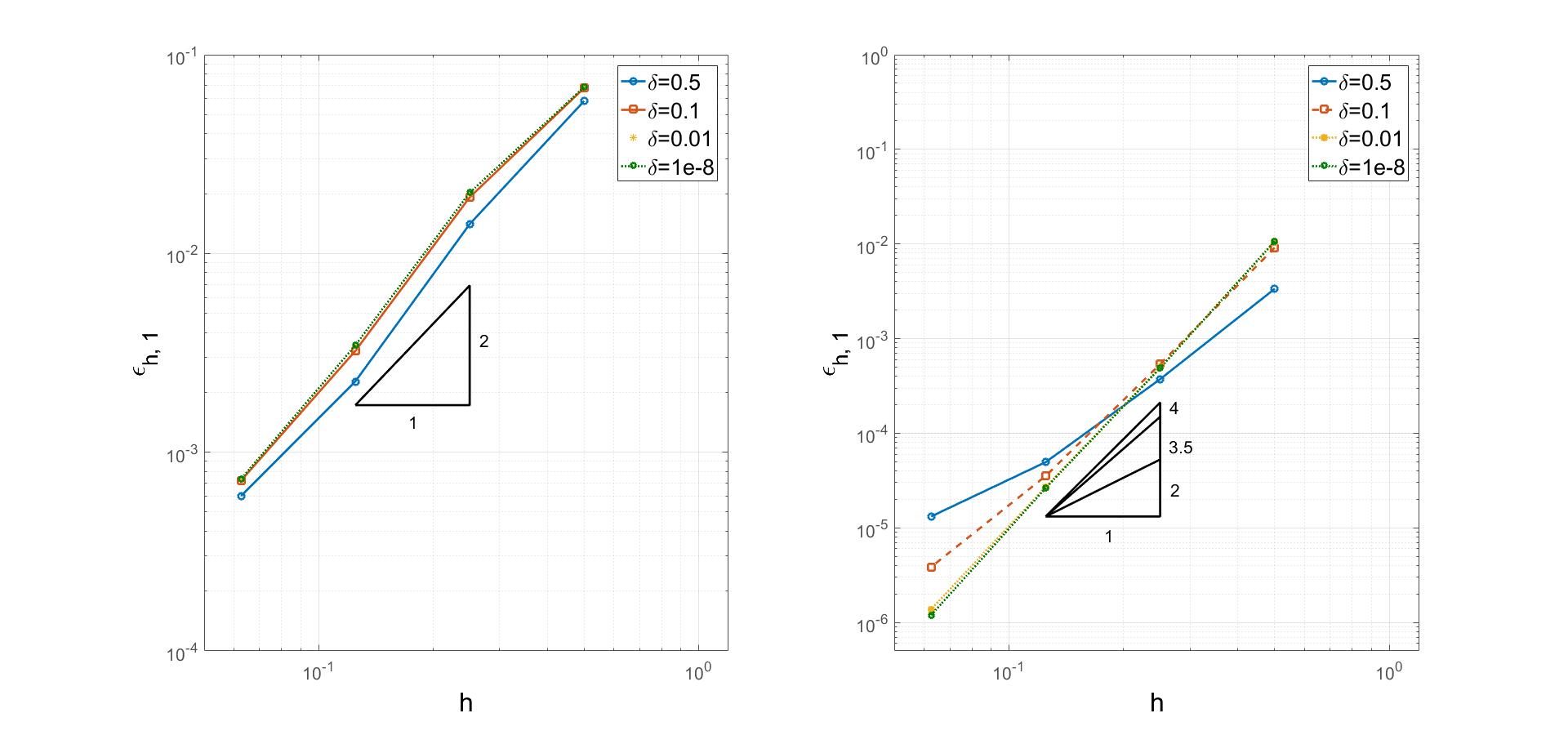} 
\caption{Test 7.3. Convergence plots for the first eigenvalue and different values of the diffusivity. Left plot: $k=1$. Right plot: $k=2$.}
\label{test3_first}
\end{figure}

\begin{figure}[!h]
\centering
\includegraphics[scale=0.2]{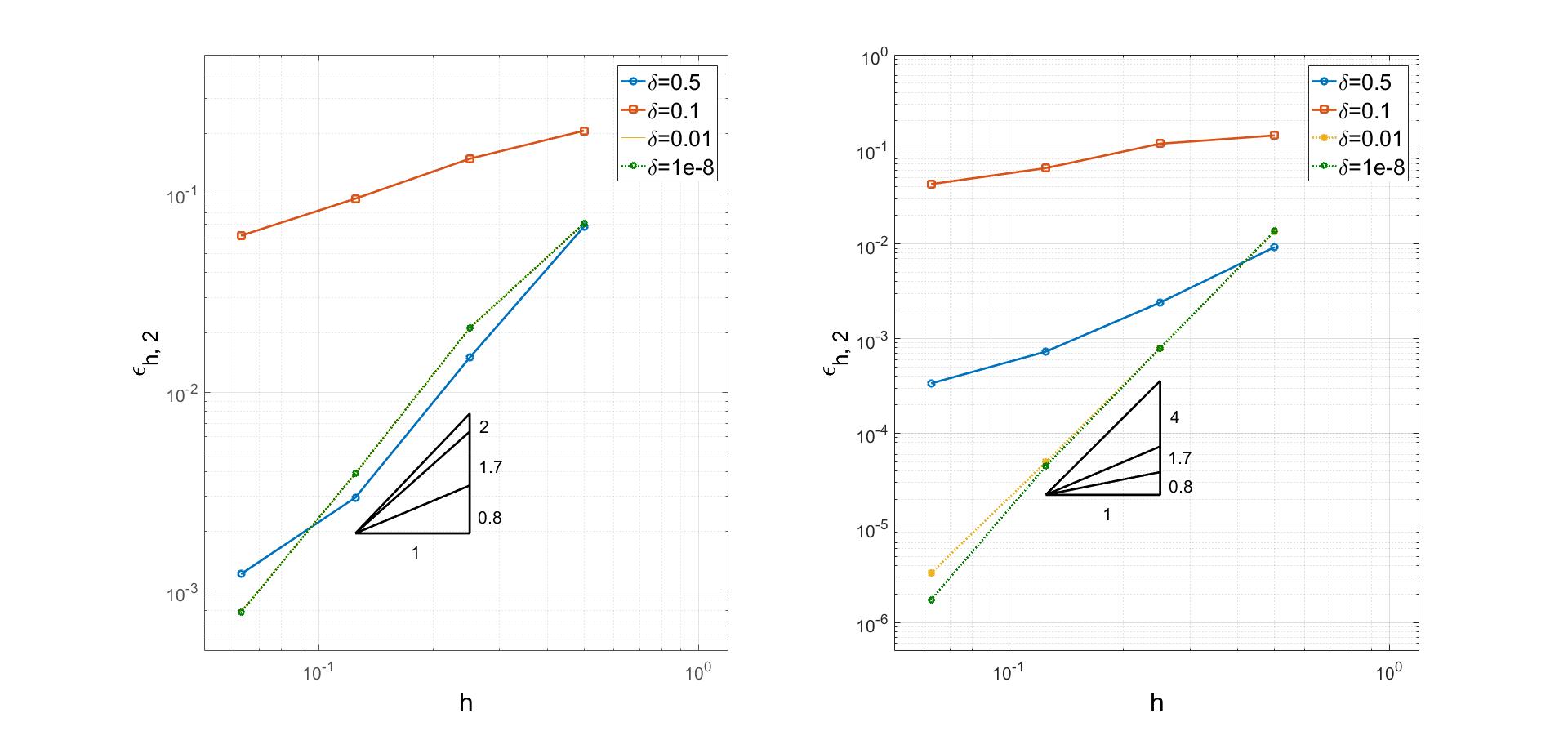} 
\caption{Test 7.3. Convergence plots for the second eigenvalue and different values of the diffusivity. Left plot: $k=1$. Right plot: $k=2$.}
\label{test3_second}
\end{figure}

We can observe, in accordance with Theorem \ref{thm:stab}, different rates of convergence that are
determined by the polynomail order of the method and by the regularity of the corresponding exact eigenfunctions \cite{dauge}. Taking this into account, the method is overall optimal, and thus stable
with respect to discontinuities in the diffusivity tensor.
\end{test}

\begin{test}
\label{test4}
In the last test we solve the eigenvalue problem  with
Neumann boundary conditions on the non-convex L-shaped domain
$\Omega = \Omega_{big} \setminus \Omega_{small}$, where $\Omega_{big}$ is the square $(-1,1)^2$ and $\Omega_{small}$ is the square $(0,1) \times (-1, 0)$.  
Also this test problem is taken from the benchmark singular
solution set \cite{dauge}. 
We apply the Virtual Element method \eqref{eq:discreteEigPbm2} with the scalar stabilization \eqref{eq:original}. 
We use the sequence of Voronoi decomposition of the domain $\Omega$ in Figure \ref{L-shaped}.
The convergence results relative to the first and the third eigenvalues are shown in Figure \ref{test4_first-third}.
For the first eigenvalue we observe a lower rate of convergence due to
the fact that the corresponding eigenfunction is in $H^{1+r}$, with
$r = 2/3 - \varepsilon$ for any $\varepsilon >0$ (see \cite{dauge}), while the third eigenfunction is analytical therefore we obtain the optimal order of convergence. The error slopes validate the predicted
convergence rates stated in Section \ref{sub:convstab}, and confirm the optimality of the
method also on non-convex domains.
\begin{figure}[!h]
\centering
\includegraphics[scale=1]{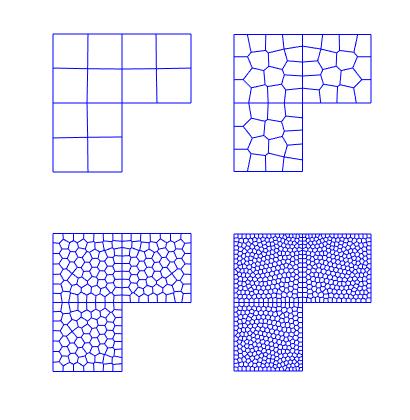} 
\caption{Test 7.4.  Adopted family of meshes for the L-shaped domain $\Omega$.} 
\label{L-shaped}
\end{figure}

\begin{figure}[!h]
\centering
\includegraphics[scale=0.2]{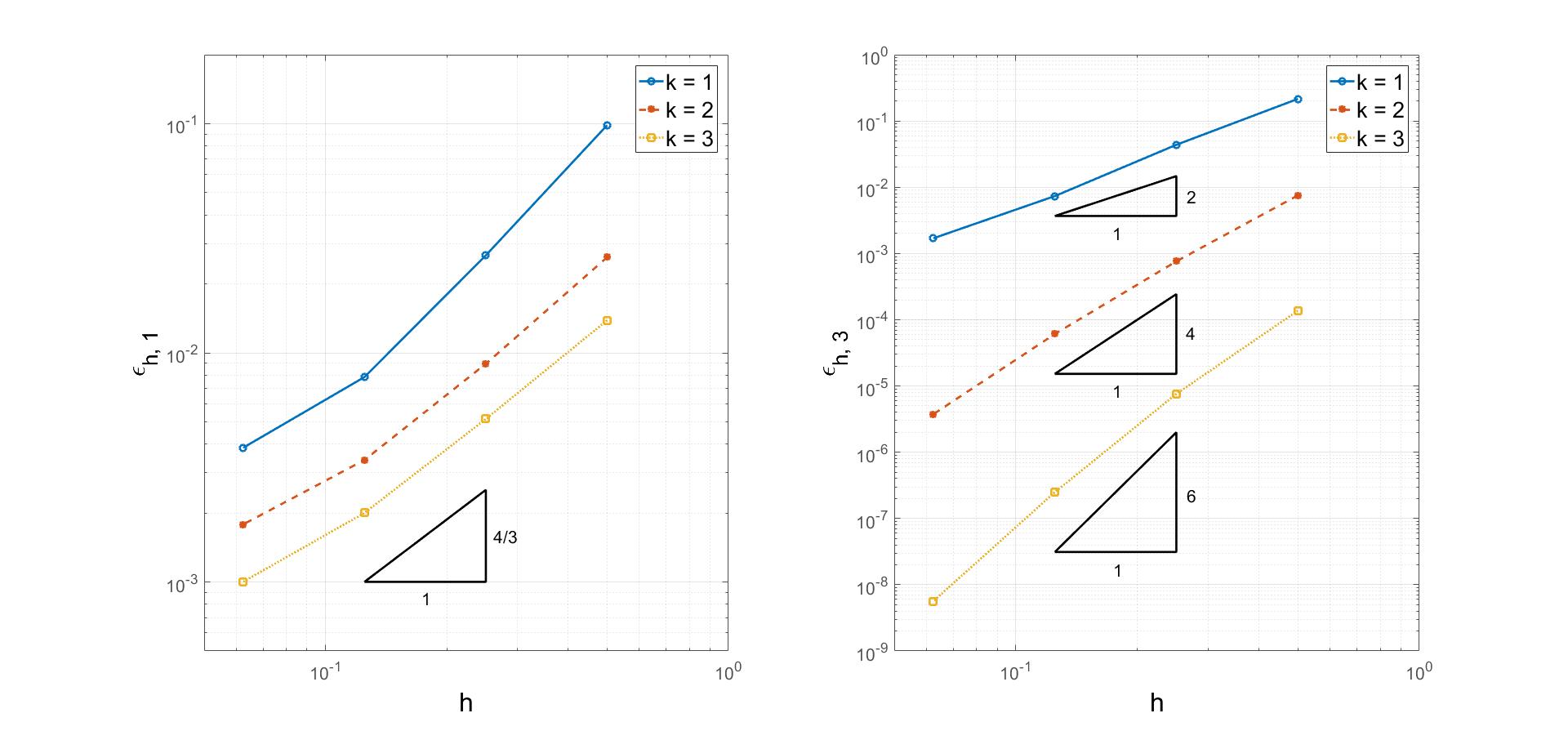} 
\caption{Test 7.4. Convergence plot for the eigenvalues for the L-shaped domain for the sequence of meshes in Figure \ref{L-shaped} and $k=1,2,3$. Left plot: first eigenvalue. Right plot: third eigenvalue.}
\label{test4_first-third}
\end{figure}

\end{test}

\section{Conclusions}
\label{sc:end}

We have analyzed the VEM approximation of elliptic eigenvalue problems. 
We proved the method is of optimal order both in the approximation of the eigenfunctions and of the 
eigenvalues. A wide set of numerical test confirm the theoretical results. 
Further development consists in studying the VEM approximation  
of eigenvalue problems in mixed form, {\itshape {a posteriori}} error estimates 
and convergence of adaptive VEM for eigenvalue problems.

\bibliographystyle{plain}
\bibliography{ref}

\end{document}